\documentclass[a4paper,twoside,reqno]{amsart}

\usepackage[foot]{amsaddr}
\usepackage[margin=1.25in]{geometry}
\usepackage{amsfonts,amsmath,amssymb,amsthm}
\usepackage{commath,mathtools,bbm}
\usepackage{bm}
\usepackage{enumitem}
\usepackage{etoolbox}
\usepackage[maxbibnames=99,style=alphabetic,doi=false,backend=biber,sorting=ynt,sortcites,natbib=true]{biblatex}
\usepackage{booktabs}
\usepackage[section]{placeins}
\usepackage[colorlinks=true,citecolor=blue,breaklinks=true]{hyperref}
\usepackage{caption,subcaption}
\usepackage{tikz}
\usetikzlibrary{decorations.pathreplacing,shapes}
\usepackage{float}
\floatstyle{ruled}
\newfloat{algorithm}{htb}{alg}
\floatname{algorithm}{Algorithm}

\usepackage{fancyhdr}
\newcommand{\myauthors}{Gupte, Kalinowski, Rigterink, Waterer}
\title[Extensions of some bilinear functions]{Extended formulations for convex hulls of \\ some bilinear functions}

\author{Akshay Gupte$^1$}
\address{$^1$School of Mathematical \& Statistical Sciences, Clemson
  University, Clemson, USA}
\thanks{Research of AG was supported by the US Office of Naval Research
  grant N00014-16-1-2725.}
\author{Thomas Kalinowski$^{2,3}$}
\address{$^{2}$School of Science and Technology, University of New England, Armidale, Australia}
\author{Fabian Rigterink$^3$}
\author{Hamish Waterer$^3$}
\address{$^3$School of Mathematical \& Physical Sciences, University of
  Newcastle, Callaghan, Australia} \thanks{Research of TK, FR, HW was supported by the ARC Linkage
  grant no. LP110200524, Hunter Valley Coal Chain Coordinator (\href{http://www.hvccc.com.au}{hvccc.com.au}) and Triple Point Technology
  (\href{http://www.tpt.com}{tpt.com}).}
\email[A.~Gupte]{agupte@clemson.edu}
\email[T.~Kalinowski]{tkalinow@une.edu.au}
\email[F.~Rigterink]{fabian.rigterink@gmail.com}
\email[H.~Waterer]{hamish.waterer@newcastle.edu.au}

\date{\today}

\theoremstyle{plain}
\newtheorem{lemma}{Lemma}
\newtheorem*{observation}{Observation}

\newtheorem{theorem}{Theorem}
\newtheorem{corollary}{Corollary}

\theoremstyle{definition}

\theoremstyle{remark}
\newtheorem{example}{Example}
\newtheorem{remark}{Remark}

\newcommand{\RR}{\mathbb{R}}
\newcommand{\reals}{\mathbb{R}}
\newcommand{\ints}{\mathbb{Z}}

\newcommand{\nats}{\mathbb{N}}
\newcommand{\eq}{\,=\,}

\renewcommand{\leq}{\leqslant}
\renewcommand{\geq}{\geqslant}
\renewcommand{\le}{\leqslant}
\renewcommand{\ge}{\geqslant}
\newcommand{\vect}[1]{\bm{#1}}
\renewcommand{\vec}[1]{\bm{#1}}
\newcommand{\f}{\psi}

\DeclareMathOperator{\conv}{conv}

\DeclareMathOperator{\cav}{cav}
\DeclareMathOperator{\vex}{vex}
\DeclareMathOperator{\LB}{LB}
\DeclareMathOperator{\UB}{UB}

\DeclareMathOperator{\sign}{sign}

\newcommand{\QP}{\mathit{QP}}

\newcommand{\cut}{\mathsf{CUT}}

\renewcommand{\P}{\mathit{P}}
\newcommand{\X}{\mathit{X}}
\newcommand{\M}{\mathit{M}}

\begin{document}

\begin{abstract}
  We consider the problem of characterizing the convex hull of the graph of a bilinear function $f$
  on the $n$-dimensional unit cube $[0,1]^n$. Extended formulations for this convex hull are
  obtained by taking subsets of the facets of the Boolean Quadric Polytope (BQP). Extending existing
  results, we propose a systematic study of properties of $f$ that guarantee that certain classes
  of BQP facets are sufficient for an extended formulation. We use a modification of Zuckerberg's
  geometric method for proving convex hull characterizations [Geometric proofs for convex hull
  defining formulations, Operations Research Letters \textbf{44} (2016), 625--629] to prove some
  initial results in this direction. In particular, we provide small-sized extended formulations for bilinear functions whose corresponding graph is either a cycle with arbitrary edge weights or a clique or an almost clique with unit edge weights.
\end{abstract}


\maketitle

\section{Introduction}
An important technique in global optimization is the construction of convex envelopes for nonconvex
functions, and there is a significant amount of literature on characterizing convex hulls of graphs
of nonlinear functions, beginning with \citep{Rikun97,Sherali97a}; see also the book
\citep{locatelli2013book}. \citeauthor{Rikun97} studies the question when this convex hull is a
polyhedron and gives a complete characterization for functions on polyhedral domains. Even if the
convex hull is a polyhedron there might be a very large number of facets, and this is reminiscent of
a situation which is quite common in combinatorial optimization: for a natural mixed integer
programming (MIP) formulation the convex hull of the feasible set can be described explicitly, but
it is a polytope whose number of facets is exponential in the instance size. One approach that has
been successful in this area is the use of {\em extended formulations} \citep{Conforti10}. The basic
idea is to introduce more variables in order to reduce the number of constraints. A reformulation of
the convex hull with a polynomial number of constraints and polynomially many additional variables
is called a {\em compact extended formulation}, and this is a key ingredient in so-called
lift-and-project methods, and other related MIP formulation techniques
\citep{Balas93,sherali1990hierarchy}. In this paper, we use a similar technique: instead of
describing the convex hull of a graph of a bilinear function in the original variable space, we seek
to describe it in a lifted space with as few inequalities as possible.

A bilinear function is a function $f : [0,1]^n \to \RR$ of the form
\[f(\vec{x}) = \sum_{1 \leq i < j \leq n} a_{ij} x_ix_j\]
with coefficients $a_{ij} \in \RR$. 
The convex hull of the graph of $f$ is the set 
\[\X(f) := \conv \{ (\vec{x},z) \in [0,1]^n \times \reals : \ z = f(\vec{x}) \},\]
which is a polytope since 
\begin{equation}\label{eq:xf2}
\X(f) = \conv \{ (\vec{x}, z) \in \{ 0, 1 \}^n \times \reals : \ z = f(\vec{x}) \}
\end{equation}
which was proved in~\citep{Rikun97,Sherali97a}. These functions arise in many problem areas; see \citep{deygupte2013pooling,gupte2016pooling,siampaper} and the references therein.

For bilinear functions, a natural setting for an extended formulation is to introduce additional
variables $y_{ij}$ representing the product $x_{i}x_{j}$ of two original variables for
$a_{ij}\neq 0$. The classical McCormick inequalities~\cite{McCormick76} for relaxing each bilinear
term are
\begin{align}
y_{ij} &\geq 0, & y_{ij} &\leq x_i, & y_{ij} &\leq x_j, & x_i+x_j-y_{ij} &\leq 1, \label{eq:mccormick} 
\end{align}
and they are exact at $0$--$1$ points, that is, they imply $y_{ij} = x_{i}x_{j}$ when
$x_{i},x_{j}\in\{0,1\}$. The McCormick relaxation is the polytope
\[\M := \left\{(\vec{x},\vec{y},z) \in [0,1]^{n(n+1)/2}\times\reals : \ z = \sum_{1 \leq i < j \leq n} a_{ij}y_{ij}, \ \eqref{eq:mccormick} \text{ for all } 1 \leq i < j \leq n \right\},\]
whose projection is typically a relaxation of $\X(f)$. The cases where the projection of $\M$ is
actually equal to $\X(f)$ have been characterized in \citep{Misener15} and independently in
\citep{Boland16}, and there are also some results in this regard for multilinear functions
\citep{Luedtke12}. In general, the McCormick relaxation can be quite weak~\cite{Boland16}, and the
purpose of this paper is to investigate extended formulations for $\X(f)$ obtained as strengthenings of the McCormick relaxation.

As is customary in the literature, let the functions $\vex[f] : [0,1]^n \rightarrow \RR$ and
$\cav[f] : [0,1]^n \rightarrow \RR$, denoting the convex and concave envelopes, respectively, of $f$
over $[0,1]^{n}$, be defined as
\begin{align*}
  \vex[f](\vec{x}) &= \min \{ z : \, (\vec{x},z) \in \X(f) \}, & \cav[f](\vec{x}) &= \max \{ z :
  \, (\vec{x},z) \in \X(f) \},
\end{align*}
so that
\[\X(f) = \{ (\vec{x},z) \in [0,1]^n \times \RR : \ \vex[f](\vec{x}) \leq z \leq \cav[f](\vec{x})\}.\]
Introducing variables $y_{ij}$ to represent the products $x_ix_j$, we are interested in describing
$\X(f)$ in terms of the $x$- and $y$-variables. To be more precise, we define a function
$\pi[f]:\reals^{n}\times\reals^{n(n-1)/2}\to\reals^{n+1}$ by
\[\pi[f](\vec{x},\vec{y})=\left(\vec{x},\sum_{1 \leq i < j \leq n}a_{ij}y_{ij}\right),\]
and extend it to the power set of $\reals^{n}\times\reals^{n(n-1)/2}$ in the usual way:
\[\pi[f](\P)=\{\pi[f](\vec{x},\vec{y}):\ (\vec{x},\vec{y})\in \P\}\]
for every $\P\subseteq\reals^{n}\times\reals^{n(n-1)/2}$. For a polytope $\P$, let the functions
$\LB_{\P}[f] : [0,1]^n \to \RR$ and $\UB_{\P}[f] : [0,1]^n \to \RR$ be defined as
\begin{alignat*}{2}
  \LB_{\P}[f](\vec{x}) &= \min \left\{ \sum_{1 \leq i < j \leq n} a_{ij}y_{ij}:\
    (\vec{x},\vec{y}) \in \P \right\} &&= \min\{ z: (\vec{x}, z) \in \pi[f](\P) \},\\
  \UB_{\P}[f](\vec{x}) &= \max \left\{ \sum_{1 \leq i < j \leq n} a_{ij}y_{ij}:\
    (\vec{x},\vec{y}) \in \P \right\} &&= \max\{ z: (\vec{x}, z) \in \pi[f](\P) \},
\end{alignat*}
respectively, so that 
\begin{equation}
\pi[f](\P) = \{ (\vec{x},z) \in [0,1]^n \times \RR : \LB_{\P}[f](\vec{x}) \leq z \leq
  \UB_{\P}[f](\vec{x}) \}.
\end{equation}
Our aim is to find a polytope $\P$ such that $\X(f) = \pi[f](\P)$, so that this $\P$ is a compact extended formulation. Observe that 
\begin{equation}\label{eq:enviff}
\X(f) = \pi[f](\P) \iff \LB_{\P}[f](\vec{x}) = \vex[f](\vec{x}), \ \UB_{\P}[f](\vec{x}) = \cav[f](\vec{x}), \ \text{ for all } \vec{x} \in [0,1]^n.
\end{equation}

There are constructive methods for deriving extended formulations of $\X(f)$ with exponentially many
variables and facet-defining inequalities, such as using the extreme point characterization in
\eqref{eq:xf2} or the nontrivial approach of using the Sherali-Adams hierarchy
\citep{sherali1990hierarchy} which can also be applied to more general nonlinear functions
\citep{ballerstein2014extended}. We restrict our attention to finding extended formulations in the
quadratic space of $(\vec{x},\vec{y})$ variables.

\citet{Padberg89} introduced the {\em Boolean Quadric Polytope} (BQP), which is the convex hull of
the binary vectors satisfying the McCormick inequalities \eqref{eq:mccormick},
\begin{equation}\label{eq:bqp}
\QP := \conv \left \{ (\vec{x}, \vec{y}) \in \{ 0, 1 \}^{n(n+1)/2}: \
    \eqref{eq:mccormick} \text{ for all } 1 \leq i < j \leq n \right \}.
\end{equation}
Since the McCormick inequalities are exact at $0$--$1$ points, we have
\[ \QP = \conv \left \{ (\vec{x}, \vec{y}) \in \{ 0, 1 \}^{n(n+1)/2}: \ y_{ij} = x_{i}x_{j} \text{ for all } 1 \leq i < j \leq n \right \}.\]  It follows from \eqref{eq:xf2} that $\QP$ is an extended formulation for $\X(f)$:
\[\X(f)=\pi[f](\QP),\]
so that  \eqref{eq:enviff}  implies $\vex[f](\vec{x}) = \LB_{\QP}[f](\vec{x})$ and
$\cav[f](\vec{x}) = \UB_{\QP}[f](\vec{x})$. In fact, \citet[Proposition 5]{burer2009nonconvex} 
showed that
\[\QP = \conv \left \{ (\vec{x}, \vec{y}) \in [ 0, 1 ]^{n(n+1)/2}: \ y_{ij} = x_{i}x_{j} \text{ for all } 1 \leq i < j \leq n \right \}.\]

\citet{Padberg89} also extended the definition of BQP in the following sense:
\[\QP(G) = \conv \left \{ (\vec{x}, \vec{y}) \in \{ 0, 1 \}^{n+m}: \ \text{\eqref{eq:mccormick} for all } ij\in E \right\},\]
where $G=(V,E)$ is the edge weighted graph associated with the bilinear function $f$. This graph has
the vertex set $V = [n] = \{ 1,\dots, n \}$, the edge set $E= \{\{i,j\} : \, a_{ij} \neq 0\}$, and
the edge weights are given by $a_{ij}$. Note that this construction gives a one-to-one
correspondence between bilinear functions and edge weighted graphs (without loops). Also,
$\QP = \QP(K_{n})$ where $K_{n}$ is the complete graph, and $\QP(G)$ is the projection of $\QP$
obtained by projecting out $y_{ij}$'s corresponding to $a_{ij}=0$. Henceforth, we express
\[ f(\vec{x}) = \sum_{ij \in E} a_{ij} x_i x_j,\] where we use $ij$ instead of $\{i,j\}$ to denote
an edge when there is no danger of ambiguity. We call an edge $ij \in E$ \emph{positive} if
$a_{ij} > 0$ and \emph{negative} if $a_{ij} < 0$. Abusing notation, we sometimes consider $\pi[f]$
as a function $\reals^{n+m}\to\reals$ where the value of $m$ is clear from the context (the number
of edges in the graph corresponding to the considered function $f$). This allows us to write
$\X(f)=\pi[f](\QP(G))$.

The polytope $\QP$, and in general $\QP(G)$, has an exponential number of facets, not all of which are known and some of the known facets are NP-hard to separate \citep{Padberg89,deza1997geometry,letchford2014new,barahona1986cut}. Furthermore, there are many graphs for which $\QP(G)$ does not have a polynomial-sized extended formulation \citep{avis2015extension}. If we do not assume any structure on $f$ and allow $f$ to be arbitrary, then a complete characterization of $\QP(G)$ seems necessary for convexifying $f$ due to the following observation.

\begin{remark}
  For every facet $\vect \alpha^T\vect x+\vect \beta^T\vect y\leq \alpha_{0}$ of $\QP(G)$ there
  exists a bilinear function $f$ such that
  $\vect \alpha^T\vect x+\vect \beta^T\vect y\leq \alpha_{0}$ is necessary to describe $\X(f)$ in
  the sense that $\pi[f](P)\supsetneq X(f)$ for the polytope $P$ obtained from $\QP$ by omitting
  $\vect \alpha^T\vect x+\vect \beta^T\vect y\leq \alpha_{0}$. To see this, just take
  $f(\vect x)=\sum_{ij\in E}\beta_{ij}x_ix_j$. Then $f(x)\leq \alpha_{0}-\vect \alpha^T\vect x$ for
  every $\vect x\in[0,1]^n$, but there exists $(\vect x^*,\vect y^*)\in P$ with
  $\vect \alpha^T\vect x^*+\vect \beta^T\vect y^*> \alpha_{0}$, and therefore
  $\pi[f](\vect x^*,\vect y^*)\not\in X(f)$.
\end{remark}

\subsection*{Our approach}
The polytope $\QP(G)$ has a very rich combinatorial structure that is not known explicitly and is
even hard to generate algorithmically. Also, since $\QP(G)$ is an extension of $\X(f)$ and two
polytopes can project onto the same polytope, it is natural to expect that for certain bilinear
functions $f$, or equivalently weighted graphs $G$, fully characterizing $\QP(G)$ may be much more
than what is actually necessary for convexifying $f$. These facts motivate us to search for graphs
$G$ for which we can identify polynomial-sized polytopes $\P\supseteq\QP(G)$ such that
$\pi[f](\P)=\X(f)$. We would also like such a $P$ to be minimal in the following sense.  An extended
formulation $\P\supseteq\QP(G)$ of $\X(f)$ is said to be \emph{minimal} if omitting any
facet-defining inequality of $\P$ leads to a polytope $\P^{\prime}\supsetneq \P$ with
$\pi[f](\P^{\prime})\supsetneq\X(f)$. In other words, we want to identify minimal classes of valid
inequalities for $\QP$ which still ensure that the polytope defined by these inequalities satisfies
$\pi[f](P)=X(f)$, the motivation being that $P\subseteq\reals^{n(n+1)/2}$ might have significantly
fewer facets than $X(f)\subseteq \reals^{n+1}$.

\subsection*{Contributions and Outline} 
We identify two graph families for which a polynomial number of commonly known valid inequalities
for $\QP(G)$ are sufficient to convexify the corresponding $f$. These main results are stated in
\textsection\ref{sec:applications}, following a short description of the valid inequalities
considered by us. Another main contribution of this paper is to use a new technique for proving the
tightness of our extended formulations. This technique is inspired by a recent work in the literature on
geometric characterization of $0$--$1$ polytopes, and is described in
\textsection\ref{sec:geometric_method}. We remove some of the technicalities of a result from the 
literature and apply our simplified result to state a description of the convex hulls of graphs of
arbitrary nonlinear, and bilinear, functions over $\{0,1\}^{n}$. This is then used to prove our main
results in \textsection\ref{sec:mainproofs}. We hope that our successful use of this technique will
inspire others to use it to prove more results on extended formulations for $\X(f)$ or other
combinatorial polytopes. 
    
\section{Background and Our Results}\label{sec:background}
To state our main results, we first need to describe the families of valid inequalities and facets
of BQP that we are interested in. Other classes of valid inequalities are also available in the
literature, many of which are obtained by exploiting the linear bijection between BQP and the cut
polytope \citep{de1990cut,deza1997geometry,barahona1986cut}. Separation questions related to some of
these inequalities have also been addressed in \citep{letchford2014new}.

\subsection{Padberg's inequalities for BQP}\label{subsec:some_facets}
The inequalities derived by \citet{Padberg89} can be written down conveniently using the following
notation, for $S\subseteq V$, $\hat{E} \subseteq E$:  
\[
  E(S) \eq \{ ij \in E: \ i, j \in S \}, \quad x(S) \eq \sum_{i \in S} x_i, \quad  y(\hat{E}) \eq \sum \limits_{ij \in \hat{E}} y_{ij}.
\]
The inequalities that are relevant for our results are the following.
\begin{description}
\item[Triangle inequalities] for any $i,j,k\in V$,
\begin{subequations}\label{eq:triangle}
\begin{align}
x_i + x_j + x_k - y_{ij} - y_{ik} - y_{jk} & \leq 1, \label{eq:triangle1} \\
- x_i + y_{ij} + y_{ik} - y_{jk} & \leq 0, \label{eq:triangle2} \\
- x_j + y_{ij} - y_{ik} + y_{jk} & \leq 0, \label{eq:triangle3} \\
- x_k - y_{ij} + y_{ik} + y_{jk} & \leq 0, \label{eq:triangle4}
\end{align}
\end{subequations}
\item[Clique inequalities] for any $S \subseteq V$ with $\lvert S\rvert \geq 3$ and integer $\alpha$,
  $1 \leq \alpha \leq \lvert S\rvert - 2$, 
\begin{align}
\alpha x(S) - y(E(S)) & \leq \frac{\alpha (\alpha + 1)}{2} \label{eq:clique1}
\end{align}
\item[Cycle inequalities] for every cycle $C \subseteq E$ and every subset $D \subseteq C$ with odd
cardinality,
  \begin{align}
x(V_0) - x(V_1) + y(C \setminus D) - y(D) & \leq (\lvert D\rvert-1)/2, \label{eq:odd_cycle}
\end{align}
where 
\begin{align*}
V_0 & = \{ u \in V: \ e \cap \hat{e} = \{u\} \text{ for some } e, \hat{e} \in D \}, \\
V_1 & = \{ u \in V: \ e \cap \hat{e} = \{u\} \text{ for some } e, \hat{e} \in C \setminus D \}.
\end{align*}
\end{description}
Using only subsets of the known facets of $\QP$, one can define different relaxations, and there
exist a number of results on conditions on $G$ which guarantee that $\QP(G)$ is equal to certain
relaxations. The LP relaxation of $\QP(G)$ is the polytope in $[0,1]^{n + |E|}$
defined by the McCormick inequalities \eqref{eq:mccormick}. \citeauthor{Padberg89} showed this is equal to $\QP(G)$ if and only if $G$ is acyclic. The \emph{cycle relaxation} of $\QP(G)$, which is a strengthening of the McCormick relaxation by adding cycle inequalities~(\ref{eq:odd_cycle}) for each chordless cycle $C\subseteq E$ and each subset $D\subseteq C$ with $\lvert D\rvert$ odd, is equal to $\QP(G)$ if and only if $G$ is $K_4$-minor-free (series-parallel graphs) \citep{de1990cut,Padberg89}. More characterizations for the cycle relaxation were obtained recently by  \citep{michini2016strength}. Another set of known results about Padberg's inequalities is that the triangle inequalities give the Chv\'atal-Gomory closure of the LP relaxation of $\QP(K_{n})$ \citep{boros1992chvatal}, and this was recently generalized to the odd cycle inequalities giving the Chv\'atal-Gomory closure of the LP relaxation of $\QP(G)$ for arbitrary $G$ \citep{Bonami2018}.

All of these literature results are about characterizing $\QP(G)$, which, as we have explained in
the introduction, might sometimes be more than what is required for convexifying $f$. One of our main 
results is that if $G$ is a chordless cycle of length $n$, which is $K_{4}$-minor-free and so the
corresponding $f$ is convexified by all the $2^{n-1}$ cycle inequalities \eqref{eq:odd_cycle}, then
only \emph{two} cycle inequalities are needed to get a polytope $\P$ with $\pi[f](\P)=\X(f)$. This
extends to cactus graphs, that is, graphs in which every edge is contained in at most one cycle. The
next section presents precise statements of our results. We clarify that none of these results implies anything about the description of $\QP(G)$.

\subsection{Main results}\label{sec:applications}
When $f(\vec{x}) = \sum_{1\leq i < j \leq n} x_{i}x_{j}$, so that the corresponding graph is the
complete graph $K_{n}$ with all edge weights equal to 1, then it is known that the McCormick
inequalities $y_{ij}\leq\min\{x_i,x_j\}$ together with the clique inequalities \eqref{eq:clique1}
with $S=V$ suffice to describe the convex hull of the graph of $f$. In fact, this implies that the
only lower bounds on the $y$-variables come from inequalities which can be written in the form
$y(E)\geq\dots sx(V)-\binom{s+1}{2}$, so the convex envelope $\vex[f]$ can be written in terms of the original variables
$x_1,\dots,x_n$ and $z=y(E)$, and this is precisely the convex envelope characterization proved
in \cite{Rikun97,Sherali97a}. We provide an alternative proof of this result in
\textsection\ref{sec:cliques}. As an extension, our first main result addresses the case of $K_{n}^{-}$,
which is the graph obtained from $K_n$ by deleting the edge between $n-1$ and $n$, with unit
weights, for which the bilinear function is
$f(\vec{x}) = \sum_{1 \leq i < j \leq n-1} x_{i}x_{j}\,+\, \sum_{i=1}^{n-2}x_{i}x_{n}$.

\begin{theorem}[Almost complete]\label{thm:clique_minus}
  If $G = K_n^-$ and all edge weights are equal to $1$, then $\X(f) = \pi[f](\P)$ where
  $\P\subseteq[0,1]^{n(n+1)/2}$ is described by $y(E)\geq 0$, the McCormick inequalities $y_{ij}\leq x_i$ and
  $y_{ij}\leq x_j$ for $1\leq i<j\leq n$, together with the $3(n-2)$ inequalities
  \begin{align}
    2x_i+x_{n-1}+x_n-y_{i,n-1}-y_{in}&\leq 2&&i=1,\dots,n-2\label{eq:clique_minus_0}\\
    s\left[x\left(V\setminus\{n-1,n\}\right)+\frac{x_{n-1}+x_n}{2}\right]\qquad\qquad\qquad
    &\nonumber\\
    -y\left(E\left(V\setminus\{n-1,n\}\right)\right)-\frac12\sum_{i=1}^{n-2}\left(y_{i,n-1}+y_{in}\right) &\leqslant \binom{s+1}{2}, && s=1,\dotsc,n-2,\label{eq:clique_minus_1}\\
    sx(V)-y(E)-y_{n-1,n} &\leqslant \binom{s+1}{2}, && s=1,\dotsc,n-2.\label{eq:clique_minus_3}
  \end{align}
Moreover, this polytope $\P$ is a minimal extension of $\X(f)$.
\end{theorem}

Constraints~\eqref{eq:clique_minus_0} are sums of McCormick inequalities
$x_i+x_{n-1}-y_{i,n-1}\leq 1$ and $x_i+x_{n}-y_{in}\leq 1$. The other inequalities in the above
theorem come from Padberg's clique inequalities:~(\ref{eq:clique_minus_1}) are the averages of the
clique inequalities for the two maximal cliques in $G$ and~(\ref{eq:clique_minus_3}) are the clique
inequalities for $K_n$. It should be of no surprise
that the non-McCormick inequalities required to describe $\X(f)$ all have negative coefficients on
the $y$ variables, because only $\vex[f](\vec{x})$ is unknown since two of the McCormick
inequalities for every $y_{ij}$ are known to describe $\cav[f](\vec{x})$ when edge weights are
non-negative (\cite{Luedtke12,Tawarmalani12}, see also Corollary~\ref{cor:cav} below). An interesting feature of this theorem is that the
variable $y_{n-1,n}$ is used although it does not correspond to an edge in $G$. The point in
Figure~\ref{fig:K5-} illustrates that the clique inequalities for the cliques in the given graph are
not sufficient in general. The point satisfies the clique inequalities for the seven 3-cliques, and
for the two $4$-cliques, as well as the cycle inequalities for all cycles in $K_5^-$, but
$\pi[f](\vect x,\vect y)=(1/2,1/2,1/2,3/4,1/4,\,3/2)\not\in X(f)$, because it violates the
inequality
\[2(x_1+x_2+x_3+x_4)+x_5-z\leq 3\]
which can be obtained by using $y_{45}\leq x_5$ to eliminate $y_{45}$ from a clique inequality for the $K_5$.

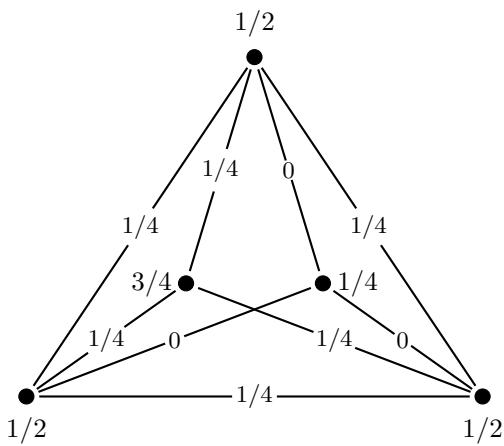
\begin{figure}[htb]
  \centering
  \begin{tikzpicture}[scale=1.5,every node/.style={circle,fill=black,draw,inner sep=2pt,outer
      sep=2pt}]
    \tikzset{l/.style={draw=none,fill=white,inner sep=.1pt,outer sep=.1pt}}
    \node[label={[label distance=-2mm]below:$1/2$}] (1) at (0,0) {};
    \node[label={[label distance=-2mm]below:$1/2$}] (2) at (4,0) {};
    \node[label={[label distance=-2mm]above:$1/2$}] (3) at (2,3) {};
    \node[label={[label distance=-2mm]left:$3/4$}] (4) at (1.4,1) {};
    \node[label={[label distance=-2mm]right:$1/4$}] (5) at (2.6,1) {};
    \draw[thick] (1) -- (2) node[midway,l] {{\small $1/4$}};
    \draw[thick] (1) -- (3) node[midway,l] {{\small $1/4$}};
    \draw[thick] (1) -- (4) node[midway,l] {{\small $1/4$}};
    \draw[thick] (1) -- (5) node[midway,l] {{\small $0$}};
    \draw[thick] (2) -- (3) node[midway,l] {{\small $1/4$}};
    \draw[thick] (2) -- (4) node[midway,l] {{\small $1/4$}};
    \draw[thick] (2) -- (5) node[midway,l] {{\small $0$}};
    \draw[thick] (3) -- (4) node[midway,l] {{\small $1/4$}};
    \draw[thick] (3) -- (5) node[midway,l] {{\small $0$}};
    \end{tikzpicture}
    \caption{A point satisfying the clique and cycle inequalities involving only variables
      corresponding to edges of $K_5^-$.}\label{fig:K5-}
  \end{figure}

From Theorem~\ref{thm:clique_minus}, we can also read off a description of $\vex[f](\vect x)$ in
terms of the original variables: $\vex[f](\vect x)=\max\{0,\,z_1,\,z_2\}$, where
\begin{align*}
  z_1 &= \max\left\{sx(V)-\min\{x_{n-1},x_n\}-\binom{s+1}{2}\,:\,1\leq s\leq n-2\right\},\\
  z_2 &=
        \max\Big\{s\left[x\left(V\setminus\{n-1,n\}\right)+\frac{x_{n-1}+x_n}{2}\right]-\binom{s+1}{2}\\
  &\qquad\qquad\qquad +\frac12\sum_{i=1}^{n-2}\max\{0,\,2x_i+x_{n-1}+x_n-2\}\,:\,1\leq s\leq n-2\Big\}.
\end{align*}

The second main result is about chordless cycles. Let $C_{n}$ denote an $n$-cycle with edges
$\{i, i + 1\}$ for $i \in [n - 1]$, and the edge $\{1, n\}$.  The vertex set and edge set are each in bijection with $[n]$ and hence indexed by $[n]$. For ease of notation, all indices have to be read modulo $n$ in the obvious way. In particular,  $n \equiv 0$ and $n+1\equiv
1$. For $i=1,\dots,n$, edge $\{i,i+1\}$ is referred to as edge $i$ and $a_{i, i + 1}$ is written as $a_i$. The index set $[n]$ can be partitioned as $[n]=E^-\cup E^+$ based on the signs of the coefficients:
\[E^- = \{i\,:\,a_i < 0\}, \quad  E^+  = \{i\:\,a_i > 0 \}.\]
We define $V^-$ (resp. $V^+$) to be the set of vertices $i\in[n]$ such that both of the edges incident with
$i$ correspond to negative (resp. positive) terms in $f(\vect x)$. More formally,
\[V^- = \{ i \in [n] : \ \{i-1,i\}\subseteq E^- \}, \quad V^+ = \{ i \in [n] : \ \{i - 1, i\}\subseteq E^+ \}. 
\]
which, in general, may not partition $[n]$.

\begin{theorem}[Cycles] \label{thm:cycles} 
For $G = C_{n}$, we have $\pi[f](\P) = \X(f)$ where
  $\P \subseteq [0, 1]^{2 n}$ is the polytope described by the McCormick
  inequalities~\eqref{eq:mccormick} and
\begin{align}
x(V^-) - x(V^+) + y(E^+) - y(E^-) & \leq \left \lfloor \frac{\lvert E^- \rvert}{2} \right \rfloor, \label{eq:cycle_1} \\
x(V^+) - x(V^-) + y(E^-) - y(E^+) & \leq \left \lfloor \frac{\lvert E^+ \rvert}{2} \right \rfloor. \label{eq:cycle_2}
\end{align}
Moreover, inequality~\eqref{eq:cycle_1} (resp.~\eqref{eq:cycle_2}) can be omitted if and only if~$\lvert E^- \rvert$
(resp.~$\lvert E^+ \rvert$) is even, and then $P$ is a minimal extension of $X(f)$.
\end{theorem}

The polytope $\P$ in Theorem~\ref{thm:cycles} is parametrized by edge weights $\vect a$ and should
be read as $\P_{\vect a}$, since \eqref{eq:cycle_1} and \eqref{eq:cycle_2} are constructed using the
sign pattern on the edge weights. Obviously, since $\P$ has only two cycle inequalities, it is a
weak relaxation of $\QP(C_{n})$, which we know is given by all the $2^{n-1}$ odd-cycle inequalities
due to $C_{n}$ being $K_{4}$-minor-free. When $\abs{E^{-}}$ is odd, inequality~\eqref{eq:cycle_1} is
Padberg's cycle inequality~\eqref{eq:odd_cycle} corresponding to the odd subset $D = E^{-}$ in the
graph $C_{n}$. This is because $D = E^{-} = E(C_{n})\setminus E^{+}$ and every vertex in $C_{n}$
having exactly two edges incident on it implies that $V_{0} = V^{-}$ and $V_{1} = V^{+}$. If
$\abs{E^{-}}$ is even, we will show in the proof of Theorem~\ref{thm:cycles} that
inequality~\eqref{eq:cycle_1} is a linear combination of McCormick inequalities. Analogous arguments
hold for $E^{+}$ and \eqref{eq:cycle_2}. Note that inequalities~\eqref{eq:cycle_1} and
\eqref{eq:cycle_2} do not use any non-edge variables $y_{ij}$; in contrast,
Theorem~\ref{thm:clique_minus} presents a minimal extension for the chordal graph $K_{n}^{-}$ using
a non-edge variable $y_{n-1,n}$.

Theorem~\ref{thm:cycles} implies that if the bilinear function corresponds to an even cycle
$C_{n}$ having both $\abs{E^{+}}$ and $\abs{E^{-}}$ even, then the McCormick inequalities are
sufficient to convexify the function. This implication is a special case of the following characterization by
\citet[Theorem 4]{Boland16}: for any bilinear function $f$, the McCormick relaxation projects onto
$\X(f)$ if and only if \emph{every} cycle in the graph $G$ of $f$ has both $\abs{E^{+}}$ and
$\abs{E^{-}}$ even. This naturally raises the question of what can be said about bilinear functions with odd cycles. In general, we should not expect to convexify a bilinear function using extended formulation for each cycle in the function, since the function may have large extension complexity, whereas any cycle has a small extended formulation given in Theorem~\ref{thm:cycles}. However, we show that for cactus graphs, which are graphs 
whose cycles are edge-disjoint, or equivalently, any two cycles have at most one common vertex, the bilinear function is convexified by considering each cycle individually.

\begin{theorem}\label{thm:cactus}
If $G$ is a cactus with $k$ cycles, then for any edge weight vector $\vec{a}$, $\X(f)$ is described by the McCormick inequalities and at most $2k$ cycle inequalities.
\end{theorem}

This result is argued using the following consequence of the method described in Section~\ref{sec:geometric_method}: if $f$ and $g$ are two bilinear functions which share at most one variable, then $X(f+g)$ can be easily described in terms of $X(f)$ and $X(g)$ (see Corollary~\ref{cor:combination} for a precise statement). 

We conclude this section by illustrating the reduction in the number of required inequalities for
small examples. In Table~\ref{tab:facet_numbers} we compare the numbers of facets of
$X(f)\subseteq\reals^{n+1}$ to the numbers of inequalities describing a polytope
$P\subseteq\reals^{n(n+1)/2}$ with $\pi[f](P)=X(f)$. The facet numbers for $X(f)$ are determined
using \texttt{polymake}~\cite{assarf2017computing}, while the numbers in the columns for the
polytopes $P$ are determined as follows.
\begin{table}[htb]
  \renewcommand{\arraystretch}{1.3}
  \centering
  \caption{Numbers of facets for $X(f)$ and for our extended formulations.}\label{tab:facet_numbers}
      \begin{tabular}{r@{\hskip .5cm}rr@{\hskip 1cm}rr@{\hskip 1cm}rr}
\toprule
& \multicolumn{2}{c}{$G=K_n$} & \multicolumn{2}{c}{$G=K_n^-$} & \multicolumn{2}{c}{$G=C_n$} \\
\cmidrule(lr){2-3}\cmidrule(lr){4-5}\cmidrule(lr){6-7}
$n$ & $\X(f)$ & $\P$ & $\X(f)$ & $\P$ & $\X(f)$ & $\P$ \\
\midrule
3 &     15 & 15 &     12 &  16 &  15 & 20 \\
4 &     36 & 24 &     34 &  27 &  26 & 26 \\
5 &    135 & 35 &    120 &  40 &  63 & 32 \\
6 &    738 & 48 &    636 &  55 &  118 & 38 \\
7 &  5,061 & 63  &  4,376 & 72 &  255 & 44 \\
8 & 40,344 & 80  & 35,372 & 91 &  498 & 50 \\ \bottomrule
    \end{tabular}  
  \end{table}
  For $G=K_n$ with unit coefficients, we can choose $P\subseteq\reals^{n(n+1)/2}$ described by the
  following $n(n+2)$ inequalities:
  \begin{itemize}
  \item $2n$ variable bounds $0\leq x_i\leq 1$ for $i\in[n]$,
  \item $n(n-1)$ McCormick upper bounds $y_{ij}\leq x_i$ and $y_{ij}\leq x_j$ for $ij\in E$,
  \item $n-1$ inequalities $y(E)\geq sx(V)-s(s+1)/2$ for $s\in[n-1]$, and
  \item one inequality $y(E)\geq 0$.    
  \end{itemize}
  For $G=K^-_n$ with unit coefficients, the following $n^2+4n-4$ inequalities are sufficient:
  \begin{itemize}
  \item $2n$ variable bounds $0\leq x_i\leq 1$ for $i\in[n]$,
  \item $n(n-1)$ McCormick inequalities $y_{ij}\leq x_i$ and $y_{ij}\leq x_j$,
  \item $3n-6$
    inequalities~\eqref{eq:clique_minus_0},~\eqref{eq:clique_minus_1},~\eqref{eq:clique_minus_3},
  \item one inequality $y(E)\geq 0$.    
  \end{itemize}
  For $G=C_n$, with arbitrary coefficients, the following $6n+2$ inequalities are sufficient:
  \begin{itemize}
  \item $2n$ variable bounds $0\leq x_i\leq 1$ for $i\in[n]$,
  \item $4n$ McCormick inequalities~\eqref{eq:mccormick}, and
  \item two inequalities~\eqref{eq:cycle_1} and \eqref{eq:cycle_2}.
  \end{itemize}

\section{A Geometric Characterization of Combinatorial Polytopes}\label{sec:geometric_method}

\subsection{Zuckerberg's method}
Zuckerberg~\cite{ZuckerbergPhD,Zuckerberg16} developed a technique to prove convex hull
characterizations for subsets of $\{0,1\}^n$. In this section, we simplify one of the technical
results from this work and extend it to the convex hulls of graphs of arbitrary functions
$\f:\{0,1\}^n\to\reals$.

We are interested in the convex hull of a set $\mathcal F\subseteq\{0,1\}^n$. Any such $\mathcal F$
can be represented as a finite combination of unions, intersections and complementations of the
sets
\[A_i=\{\vect\xi\in\{0,1\}^n\,:\,\xi_i=1\},\quad i=1,\dots,n,\]
  and we fix such a representation $F(A_1,\dots,A_n)$.  For
instance, the set $\X(f)$ for the function $f:[0,1]^2\to[0,1]$ given by $f(x_1,x_2)=x_1x_2$ is the
convex hull of the set
\[\mathcal F=\{(0,0,0),\,(0,1,0),\,(1,0,0),\,(1,1,1)\},\]
which is given by the variable bounds $0\leq x_1,x_2,x_3\leq 1$ and the McCormick inequalities
\begin{align*}
  x_3 &\leq x_1, & x_3&\leq x_2, & x_3&\geq x_1+x_2-1.
\end{align*}
In this case $\mathcal F$ can be represented as
\begin{multline}\label{eq:zucker_example}
\mathcal F=F(A_1,A_2,A_3)=\left(A_1\cap A_2\cap A_3\right)\cup\left(\overline{A_1\cup A_2\cup
    A_3}\right)\cup\left(A_1\cap\overline{A_2\cup A_3}\right)\\
\cup\left(A_2\cap\overline{A_1\cup A_3}\right), 
\end{multline}
where $\overline{\,\cdot\,}$ indicates the complement in $\{0,1\}^3$.

The main result in \citep[Theorem 7]{Zuckerberg16} can be stated as follows. For every $\vect
x\in[0,1]^n$, we have $\vect x\in\conv(\mathcal F)$ if and only if there exist
\begin{itemize}
\item a set $U$ with a collection $\mathcal L$ of subsets which contains the empty set and is closed under taking complements and finite unions, and
\item a function $\mu:\mathcal L\to\reals$ with $\mu(U)=1$ and $\mu(L_1\cup\dotsb\cup
  L_k)=\mu(L_1)+\dotsb+\mu(L_k)$ for any finite collection of pairwise disjoint elements
  $L_1,\dotsc,L_k\in\mathcal L$, and
\item sets $X_1,\dotsc,X_n\in \mathcal L$
with $\mu(X_i)=x_i$ for all $i\in[n]$ and $\mu(F(X_1,\dots,X_n))=1$ (where the complement has to be
taken in $U$ instead of $\{0,1\}^n$).
\end{itemize}
To be precise, Zuckerberg states only one direction of this equivalence, but the other one is easy
(see the proof of Theorem~\ref{thm:zucker} below).
\begin{example}\label{ex:zucker}
For
$\mathcal F=\{(0,0,0),\,(0,1,0),\,(1,0,0),\,(1,1,1)\}$ we take $U$ to be the half-open
interval $[0,1)$ with $\mathcal L$ being the collection of unions of finitely many half-open
intervals and $\mu$ the Lebesgue measure (restricted to $\mathcal L$), that is,
  \begin{align}
    \mathcal L &= \left\{[a_1,b_1)\cup\dotsb\cup[a_k,b_k)\ :\ 0\leq a_1<b_1<a_2<b_2<\dotsb<a_{k}<b_k\leq 1,\
        k\in\nats\right\},\label{eq:intervals}\\
    \mu(X) &= (b_1-a_1)+\dotsb+(b_k-a_k)\qquad\qquad \text{for }X=[a_1,b_1)\cup\dotsb\cup[a_k,b_k)\in\mathcal L.\label{eq:measure}
  \end{align}
For $(x_1,x_2,x_3)\in\conv(\mathcal F)$ we set
\begin{align*}
 X_1&=[0,\,x_1), & X_2 &= [x_1-x_3,\,x_1+x_2-x_3), & X_3 &= [x_1-x_3,\,x_1). 
\end{align*}
Then $\mu(X_i)=x_i$ for all $i\in\{1,2,3\}$, and from
\begin{multline*}
F(X_1,X_2,X_3)= \left(X_1\cap X_2\cap X_3\right)\cup\left(\overline{X_1\cup X_2\cup
    X_3}\right)\cup\left(X_1\cap\overline{X_2\cup X_3}\right)
\cup\left(X_2\cap\overline{X_1\cup X_3}\right)\\
=[x_1-x_3,x_1)\cup[x_1+x_2-x_3,1)\cup[0,x_1-x_3)\cup[x_1,x_1+x_2-x_3)=[0,1) 
\end{multline*}
we get $\mu\left(F(X_1,X_2,X_3)\right)=1$, as required. This provides a proof that the McCormick
inequalities indeed characterize the convex hull of the set
$\{(x_1,\,x_2,\,x_1x_2)\,:\,x_1,x_2\in[0,1]\}$. 
\begin{figure}[htb]
  \centering
    \begin{tikzpicture}[xscale=8,yscale=1]
\draw[thick] (0,0) -- (1,0);
\draw[thick] (0,.1) -- (0,.-.1);
\draw (.1,.05) -- (.1,-.05);
\draw (.2,.05) -- (.2,-.05);
\draw (.3,.05) -- (.3,-.05);
\draw (.4,.05) -- (.4,-.05);
\draw[thick] (.5,.1) -- (.5,.-.1);
\draw (.6,.05) -- (.6,-.05);
\draw (.7,.05) -- (.7,-.05);
\draw (.8,.05) -- (.8,-.05);
\draw (.9,.05) -- (.9,-.05);
\draw[thick] (1,.1) -- (1,.-.1);
\node at (0,-.4) {$0$};
\node at (1,-.4) {$1$};
\node at (.5,-.4) {$0.5$};
 \draw[fill=lightgray] (0,.2) rectangle (.5,.6);
 \draw[fill=lightgray] (.4,.7) rectangle (.8,1.1);
 \draw[fill=lightgray] (.4,1.2) rectangle (.5,1.6); 
 \node at (-.1,.4) {$X_1$};
 \node at (-.1,.9) {$X_2$};
 \node at (-.1,1.4) {$X_3$}; 
  \end{tikzpicture}  
  \caption{The sets $X_1$, $X_2$ and $X_3$ for $x_1=0.5$, $x_2=0.4$ and $x_3=0.1$ with
    $X_1\cap X_2\cap X_3=[0.4,0.5)$, $\overline{X_1\cup X_2\cup X_3}=[0.8,1)$,
    $X_1\cap\overline{X_2\cup X_3}=[0,0.4)$ and $X_2\cap\overline{X_1\cup X_3}=[0.5,0.8)$\,. }
  \label{fig:example_zucker}
\end{figure}
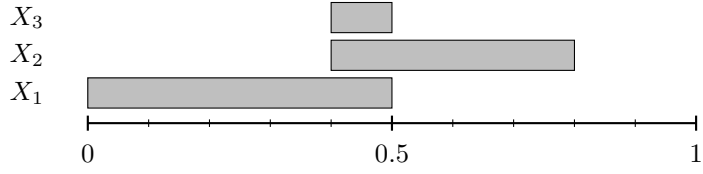
The construction is illustrated in Figure~\ref{fig:example_zucker} for $\vect x=(0.5,0.4,0.1)$. The
sets $X_1$, $X_2$ and $X_3$ do not only provide a certificate that $\vect x\in\conv\mathcal F$, but
they also encode a representation of $\vect x$ as a convex combination of the elements of
$\mathcal F$. To see this we associate with each $t\in[0,1)$ the vector $\vect
x(t)=(x_1(t),\,x_2(t),\,x_3(t))$ with $x_i(t)=1$ if $t\in X_i$ and $x_i(t)=0$ if $t\not\in X_i$. In our
example
\[\vect x(t)=
  \begin{cases}
    (1,\,0,\,0) &\text{for }t\in X_1\cap\overline{X_2\cup X_3}=[0,0.4),\\
    (1,\,1,\,1) &\text{for }t\in X_1\cap X_2\cap X_3=[0.4,0.5),\\
    (0,\,1,\,0) &\text{for }t\in X_2\cap\overline{X_1\cup X_3}=[0.5,0.8),\\
    (0,\,0,\,0) &\text{for }t\in \overline{X_1\cup X_2\cup X_3}=[0.8,1),
  \end{cases}
\]
which corresponds to the convex representation
\begin{equation}\label{eq:example_convex_comb}
  \begin{pmatrix}
    0.5\\ 0.4\\ 0.1
  \end{pmatrix} = 0.4
  \begin{pmatrix}
    1\\ 0\\ 0
  \end{pmatrix} + 0.1
  \begin{pmatrix}
    1\\ 1\\ 1
  \end{pmatrix} + 0.3
  \begin{pmatrix}
    0\\ 1\\ 0
  \end{pmatrix} +
  0.2
  \begin{pmatrix}
    0 \\ 0 \\ 0
  \end{pmatrix}.  
\end{equation}
\end{example}

Our main simplification of Zuckerberg's result is that the statement remains true if $U$ and
$\mathcal L$ are fixed as in Example~\ref{ex:zucker}. As a consequence, the condition
$\mu(F(X_1,\dots,X_n))=1$ can be replaced by $F(X_1,\dots,X_n)=U$, and in fact, the set theoretic
representation of $\mathcal F$ can be avoided completely using the standard identification of the
elements of $\{0,1\}^n$ with subsets of $[n]$.  More precisely, a vector $\vect\xi\in\{0,1\}^n$ is
identified with the set $\{i\in[n]\,:\,\xi_i=1\}$, so that in particular the elements of
$\mathcal F$ are identified with subsets of $[n]$. The following theorem is a reformulation of
\citep[Theorem 7]{Zuckerberg16} which is more convenient for our purpose. We provide a complete
proof in our setting because the proof is short and it would be a bit cumbersome to explain in
detail how our variant can be obtained from the arguments in~\cite{Zuckerberg16}. Following the proof of the theorem
we explain in detail how Zuckerberg's original statement can be obtained as a consequence of our
theorem (see Remark~\ref{rem:orig_zucker}).
\begin{theorem}\label{thm:zucker}
  Let $\mathcal F\subseteq\{0,1\}^n$, $\vect x\in[0,1]^n$, and let $\mathcal L$ and
  $\mu$ be defined by~\eqref{eq:intervals} and~\eqref{eq:measure}. Then
  $\vect x\in\conv(\mathcal F)$ if and only if there are sets $X_1,\dotsc,X_n\in\mathcal L$ such
  that $\mu(X_i)=x_i$ for all $i\in[n]$, and $\{i\in[n]\,:\,t\in X_i\}\in\mathcal F$
  for every $t\in[0,1)$.
\end{theorem}
\begin{proof}
  Let us fix an ordering $\vect\xi^1,\dotsc,\vect\xi^{\lvert\mathcal F\rvert}$ of $\mathcal F$ (In
  Example~\ref{ex:zucker}, the ordering would be the one in which the elements of $\mathcal F$
  appear in the convex combination~\eqref{eq:example_convex_comb}), and
  suppose $\vect x\in\conv(\mathcal F)$, say
  $\vect x=\lambda_1\vect\xi^1+\dotsb+\lambda_{\lvert\mathcal F\rvert}\vect\xi^{\lvert\mathcal
    F\rvert}$ with $\lambda_1+\dots+\lambda_{\lvert\mathcal F\rvert}=1$ and $\lambda_k\geq 0$ for all
  $k\in[\lvert\mathcal F\rvert]$. We define a partition $U=I_1\cup\dotsb\cup I_{\lvert\mathcal F\rvert}$ by setting
  $I_1=[0,\lambda_1)$ and $I_k=\left[\lambda_1+\dotsb+\lambda_{k-1},\,\lambda_1+\dotsb+\lambda_{k}\right)$
  for $k\in\{2,\dotsc,\lvert\mathcal F\rvert\}$. For
  \[X_i=\bigcup_{k\,:\,\xi^k_i=1}I_k,\]
  we have, for every $i\in[n]$,
  \[\mu(X_i)=\sum_{k\,:\,\xi^k_i=1}\mu(I_k)=\sum_{k\,:\,\xi^k_i=1}\lambda_k=\sum_{k=1}^{\lvert\mathcal
      F\rvert}\lambda_k\xi^k_i=x_i,\]
  and, for every $t\in[0,1)$, there is a unique index $k$ with $t\in I_k$, and then
  \[\{i\in[n]\,:\,t\in X_i\}=\{i\in[n]\,:\,\xi^k_i=1\}=\vect\xi^k\in\mathcal F,\]
  as required. Conversely, if $X_i$ are sets with the described properties,
  we can set
  \[\lambda_k=\mu\left(\left\{t\ :\ \{i\in[n]\ :\ t\in X_i\}=\vect\xi^k\right\}\right)\]
  for $k=1,\dotsc,\lvert\mathcal F\rvert$ to obtain the required convex representation $\vect
  x=\lambda_1\vect\xi^1+\dotsb+\lambda_{\lvert\mathcal F\rvert}\vect\xi^{\lvert\mathcal
    F\rvert}$. To see this note that by assumption, for every $t\in[0,1)$, there is a unique
  $k\in[\lvert\mathcal F\rvert]$ with $\vect\xi^k=\{i\in[n]\,:\,\xi^k_i=1\}$, and then for every $j\in[n]$,
  $t\in X_j$ if and only if $\xi^k_j=1$. In other words,
  \[X_j=\bigcup_{k\,:\,\xi^k_j=1}\left\{t\ :\ \{i\in[n]\ :\ t\in X_i\}=\vect\xi^k\right\},\]
  and using this we can verify that
  $\lambda_1\vect\xi^1+\dotsb+\lambda_{\lvert\mathcal F\rvert}\vect\xi^{\lvert\mathcal
    F\rvert}=\vect x$: for every $j\in[n]$,
  \begin{multline*}
    \sum_{k=1}^n\lambda_k\xi^k_j=\sum_{k=1}^n\mu\left(\left\{t\ :\ \{i\in[n]\,:\,t\in
        X_i\}=\vect\xi^k\right\}\right)\xi^k_j\\
    =\sum_{k\,:\,\xi^k_j=1}\mu\left(\left\{t\ :\ \{i\in[n]\,:\,t\in X_i\}=\vect\xi^k\right\}\right)
    =\mu\left(\bigcup_{k\,:\,\xi^k_j=1}\left\{t\ :\ \{i\in[n]\ :\ t\in
        X_i\}=\vect\xi^k\right\}\right)\\
    =\mu(X_j)=x_j.\qedhere
  \end{multline*}
\end{proof}
\begin{remark}\label{rem:orig_zucker}
  Theorem~\ref{thm:zucker} implies Zuckerberg's criterion since for any subsets
  $X_1,\dots,X_n\subseteq[0,1)$,
  \[F(X_1,\dots,X_n)=\{t\in[0,1)\,:\,\{i\in[n]\,:\,t\in X_i\}\in\mathcal F\}.\]
  This can be seen by induction on the structure of the formula $F$. The base case is $\mathcal
  F=A_j$ for some $j\in[n]$. Then
  \begin{multline*}
    F(X_1,\dots,X_n)=X_j=\{t\in[0,1)\,:\,j\in\{i\in[n]\,:\,t\in
    X_i\}\}\\
=\{t\in[0,1)\,:\,\{i\in[n]\,:\,t\in X_i\}\in A_j\},
  \end{multline*}
  as required. For the induction step we have either
  \begin{enumerate}
  \item $\mathcal F=F(A_1,\dots,A_n)=\overline{F_1(A_1,\dots,A_n)}$, or
  \item $\mathcal F=F(A_1,\dots,A_n)=F_1(A_1,\dots,A_n)\cup F_2(A_1,\dots,A_n)$, or 
  \item $\mathcal F=F(A_1,\dots,A_n)=F_1(A_1,\dots,A_n)\cap F_2(A_1,\dots,A_n)$,
  \end{enumerate}
 and in each case we can verify the statement.
  \begin{description}
  \item[Case 1] $\mathcal F=\overline{F_1(A_1,\dots,A_n)}$. By induction,
    \[F_1(X_1,\dots,X_n)=\{t\in[0,1)\,:\,\{i\in[n]\,:\,t\in X_i\}\in F_1(A_1,\dots,A_n)\},\]
    and then
    \begin{multline*}
      t\in F(X_1,\dots,X_n)\iff t\not\in F_1(X_1,\dots,X_n)\iff \{i\in[n]\,:\,t\in X_i\}\not\in
      F_1(A_1,\dots,A_n)\\
\iff      \{i\in[n]\,:\,t\in X_i\}\in
      \overline{F_1(A_1,\dots,A_n)}=\mathcal F.
    \end{multline*}
  \item[Case 2] $\mathcal F=F_1(A_1,\dots,A_n)\cup F_2(A_1,\dots,A_n)$. By induction,
    \[F_k(X_1,\dots,X_n)=\{t\in[0,1)\,:\,\{i\in[n]\,:\,t\in X_i\}\in F_k(A_1,\dots,A_n)\},\]
    for $k\in\{1,2\}$, and then
    \begin{multline*}
      t\in F(X_1,\dots,X_n)\iff t\in F_1(X_1,\dots,X_n)\cup F_2(X_1,\dots,X_n)\\
      \iff \{i\in[n]\,:\,t\in X_i\}\in F_1(A_1,\dots,A_n)\cup F_2(A_1,\dots,A_n)=\mathcal F.
    \end{multline*}
  \item[Case 3] $\mathcal F=F_1(A_1,\dots,A_n)\cap F_2(A_1,\dots,A_n)$. By induction,
    \[F_k(X_1,\dots,X_n)=\{t\in[0,1)\,:\,\{i\in[n]\,:\,t\in X_i\}\in F_k(A_1,\dots,A_n)\},\]
    for $k\in\{1,2\}$, and then
    \begin{multline*}
      t\in F(X_1,\dots,X_n)\iff t\in F_1(X_1,\dots,X_n)\cap F_2(X_1,\dots,X_n)\\
      \iff \{i\in[n]\,:\,t\in X_i\}\in F_1(A_1,\dots,A_n)\cap F_2(A_1,\dots,A_n)=\mathcal F.
    \end{multline*}
  \end{description}
\end{remark}

Next we extend the statement of Theorem~\ref{thm:zucker} so that it applies to the convex hull $X(\f)$ of the
graph of an arbitrary function $\f:\{0,1\}^n\to\reals$. For sets $X_1,\dotsc,X_n\in\mathcal L$ we
partition $U$ into $2^n$ subsets $R_{\vect \xi}(X_1,\dotsc,X_n)$, $\vect\xi\in\{0,1\}^n$, defined by
\[R_{\vect \xi}(X_1,\dotsc,X_n)=\{t\in[0,1)\,:\,\{i\in[n]\ :\ t\in X_i\}=\vect\xi\}.\]
Let us define two functions $\f_-,\f_+:[0,1]^n\to\reals$ as
\begin{align*}
  \f_-(\vect x) &= \min\left\{\sum_{\vect\xi\in\{0,1\}^n}\mu(R_\xi(X_1,\dotsc,X_n))
                 \f(\vect\xi)\,:\,X_i\in\mathcal L,\
                 \mu(X_i)=x_i\text{ for all }i\in[n] \right\},\\
  \f_+(\vect x) &= \max\left\{\sum_{\vect\xi\in\{0,1\}^n}\mu(R_\xi(X_1,\dotsc,X_n))
                 \f(\vect\xi)\,:\,X_i\in\mathcal L,\
                 \mu(X_i)=x_i\text{ for all }i\in[n] \right\}.
\end{align*}

\begin{theorem}\label{thm:zucker_function}
  For every function $\psi:\{0,1\}^n\to\reals$,
  \[X(\f)=\left\{(\vect x,z)\in[0,1]^n\times\reals :\ \f_-(\vect x)\leqslant z\leqslant \f_+(\vect
      x)\right\}.\]  
\end{theorem}
\begin{proof}
First suppose $(\vect x,z)\in X(\f)$, say
\[(\vect x,\,z)=\sum_{k=1}^{2^n}\lambda_k\left(\vect\xi^k,\,\f\left(\vect\xi^k\right)\right),\]
where $\vect\xi^1,\dots,\vect\xi^{2^n}$ is a fixed ordering of $\{0,1\}^n$. The sets $X_i\in
\mathcal L$ with $\mu(X_i)=x_i$  are defined exactly as in the proof of
Theorem~\ref{thm:zucker}: For the partition $U=I_1\cup\dotsb\cup I_{2^n}$ with
  $I_1=[0,\lambda_1)$ and $I_k=\left[\lambda_1+\dotsb+\lambda_{k-1},\,\lambda_1+\dotsb+\lambda_{k}\right)$
  for $k\in\{2,\dotsc,2^n\}$, we set
  \[X_i=\bigcup_{k\,:\,\xi^k_i=1}I_k.\]
  For every $k\in[2^n]$, $R_{\vect\xi^k}(X_1,\dotsc,X_n)=I_k$,
  hence $\mu(R_{\vect\xi^k}(X_1,\dotsc,X_n))=\lambda_k$. With
\[z=\sum_{k=1}^{2^n}\lambda_k\psi\left(\vect\xi^k\right)=\sum_{k=1}^{2^n}\mu\left(R_{\vect\xi^k}(X_1,\dotsc,X_n)\right)\psi\left(\vect\xi^k\right)\]
it follows that $\f_-(\vect x)\leqslant z\leqslant \f_+(\vect x)$.

For the converse, suppose
  $\f_-(\vect x)\leqslant z\leqslant \f_+(\vect x)$, and let $(X_1,\dotsc,X_n)$ and
  $(X'_1,\dotsc,X'_n)$ be optimizers for the problems defining $\f_-(\vect x)$ and $\f_+(\vect x)$,
  respectively. We write $z=t\f_-(\vect x)+(1-t)\f_+(\vect x)$ for some $t\in[0,1]$, and set
  \[\lambda(\vect\xi)=t\mu( R_{\vect\xi}(X_1,\dotsc,X_n)) + (1-t)\mu( R_{\vect\xi}(X'_1,\dotsc,X'_n))\]
  for all $\vect\xi\in\{0,1\}^n$. This gives the required convex representation
  \[(\vect x,z)=\sum_{\vect\xi\in\{0,1\}^n}\lambda(\vect\xi)(\vect\xi,\f(\vect\xi)).\qedhere\]
\end{proof}

For the particular case that the function $\f$ has the form $\f(\vect x)=f(\vect x)=\sum_{ij\in E}a_{ij}x_ix_j$, Theorem~\ref{thm:zucker_function} yields the following interpretation for $\cav[f](\vect x)$ and $\vex[f](\vect x)$.

\begin{corollary}\label{cor:set_interpretation}
For the bilinear function $f(\vect x)$, we have
  \begin{align*}
    \vex[f](\vect x)&=\min\left\{\sum_{ij\in E}a_{ij}\mu(X_i\cap X_j)\ :\ X_i\in\mathcal L,\
                      \mu( X_i)=x_i\text{ for all }i\in[n]\right\},\\
    \cav[f](\vect x)&=\max\left\{\sum_{ij\in E}a_{ij}\mu(X_i\cap X_j)\ :\ X_i\in\mathcal L,\
                      \mu(X_i)=x_i\text{ for all }i\in[n]\right\}.
  \end{align*}
  In particular, for a polytope $P\subseteq\reals^{n(n+1)/2}$ with $P\subseteq X(f)$, we have
  $\pi[f](P)=X(f)$ if and only if for every $\vect x\in[0,1]^n$, there exist
  $X_1,\dots,X_n\in\mathcal L$ and $X'_1,\dots,X'_n\in\mathcal L$ with $\mu(X_i)=\mu(X_i')=x_i$ for
  all $i\in[n]$, and
\begin{align*}
  \sum_{ij\in E}a_{ij}\mu(X_i\cap X_j) &= \LB_P[f](\vect x), & \sum_{ij\in E}a_{ij}\mu(X'_i\cap X'_j) &= \UB_P[f](\vect x).
\end{align*}

\end{corollary}
\begin{proof}
  We observe that
  \begin{align*}
    \sum_{\vect\xi\in\{0,1\}^n}&\mu\left( R_{\vect\xi}(X_1,\dotsc,X_n)\right)
    f(\vect\xi) = \sum_{\vect\xi\in\{0,1\}^n}\mu\left( R_{\vect\xi}(X_1,\dotsc,X_n)\right)
                  \sum_{ij\in E}a_{ij}\xi_i\xi_j\\
                &=\sum_{ij\in E}a_{ij}\sum_{\vect\xi\in\{0,1\}^n}\mu\left(
                  R_\xi(X_1,\dotsc,X_n)\right)\xi_i\xi_j \\
                &=\sum_{ij\in E}a_{ij}\sum_{\vect\xi\in\{0,1\}^n\,:\,\xi_i=\xi_j=1}\mu\left( \{t\in[0,1)\ :\
                  \{k\in[n]\ :\ t\in X_k\}=\vect\xi\}\right)\\
                &=\sum_{ij\in E}a_{ij}\mu\left(\bigcup_{\vect\xi\in\{0,1\}^n\,:\,\xi_i=\xi_j=1} \{t\in[0,1)\ :\
                  \{k\in[n]\ :\ t\in X_k\}=\vect\xi\}\right)\\
                &=\sum_{ij\in E}a_{ij}\mu\left( X_i\cap X_j\right).\qedhere
  \end{align*}
\end{proof}
\begin{example}\label{ex:mccormick_again}
  We can use Corollary~\ref{cor:set_interpretation} to prove again that $X(f)$ for
  $f(x_1,x_2)=x_1x_2$ is given by the McCormick inequalities. From
  \[\mu(X_1\cap X_2)\leq\min\{\mu(X_1),\mu(X_2)\}=\min\{x_1,x_2\},\]
  it follows that $\cav[f](x)\leq\min\{x_1,x_2\}$, and with $X_1=[0,x_1)$, $X_2=[0,x_2)$ we see that
  this bound can be achieved for all $x_1,x_2\in[0,1]$, and therefore the concave envelope is given
  by $x_3\leq x_1$ and $x_3\leq x_2$. Similarly, from
  \[\mu(X_1\cap X_2)\geq\max\{0,\mu(X_1)+\mu(X_2)-1\}=\max\{0,x_1+x_2-1\},\]
  it follows that $\vex[f](x)\geq\min\{0,x_1+x_2-1\}$, and with $X_1=[0,x_1)$, $X_2=[1-x_2,1)$ we see that
  this bound can be achieved for all $x_1,x_2\in[0,1]$, and therefore the convex envelope is given
  by $x_3\geq 0$ and $x_3\geq x_1+x_2-1$.
\end{example}
This example illustrates the use of Corollary~\ref{cor:set_interpretation} to characterize
$X(f)$ in the simplest possible case, the function $f(x_1,x_2)=x_1x_2$. It contains the main
idea of the proofs of our main results presented in the next section. The difference is, that for
more complicated functions the choice of the sets $X_i$ is not obvious, and can depend on the
particular vector $\vect x$ for which we want to find $\vex[f](\vect x)$.
\begin{example}\label{ex:convex_combination}
  For $n=5$ consider the function
  \[f(\vect x)=\sum_{1\leq i<j\leq 5}x_ix_j,\] corresponding to the graph shown on the left in
  Figure~\ref{fig:convex_combination}, and the point $\vect
  x=(0.6,\,0.3,\,0.3,\,0.9,\,0.4)$. Corollary~\ref{cor:set_interpretation} can be used to verify
  that $\vex[f](\vect x)\leq 2$. For this purpose, consider the sets
  \begin{align*}
    X_1 &= [0,0.6), & X_2 &= [0.6,0.9), & X_3 &=[0,0.2)\cup[0.9,1),\\
    X_4 &= [0,0.1)\cup[0.2,1), & X_5 &= [0.1,0.5),
  \end{align*}
  illustrated on the right in Figure~\ref{fig:convex_combination}.
   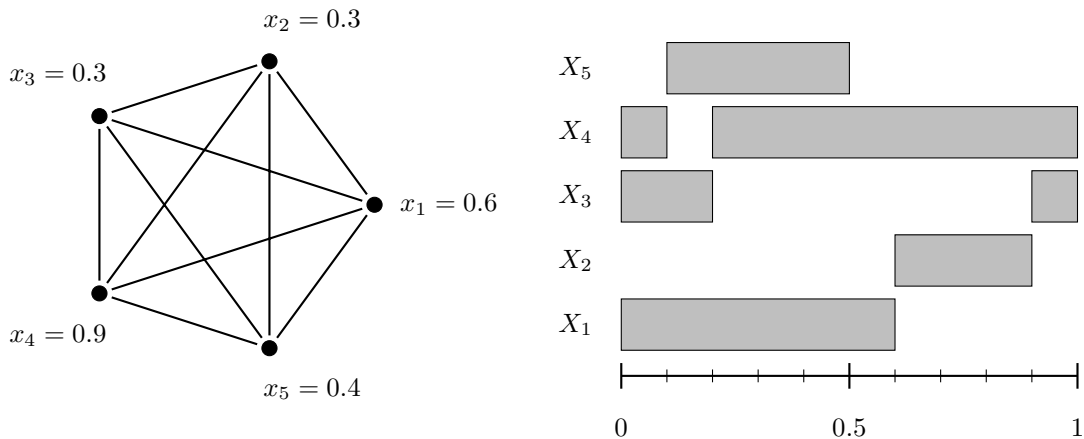
\begin{figure}[htb]
  \begin{minipage}[b]{.49\linewidth}
\centering
  \begin{tikzpicture}[scale=2,every node/.style={circle,fill=black,draw,inner sep=2pt,outer
      sep=2pt}]
\node[label={right:$x_1=0.6$}] (1) at (0:1) {};
\node[label={[label distance=-.2cm]72:$x_2=0.3$}] (2) at (72:1) {};
\node[label={[label distance=-.2cm]144:$x_3=0.3$}] (3) at (144:1) {};
\node[label={[label distance=-.2cm]-144:$x_4=0.9$}] (4) at (-144:1) {};
\node[label={[label distance=-.2cm]-72:$x_5=0.4$}] (5) at (-72:1) {};
\draw[thick] (1) -- (2) -- (3) -- (4) -- (5) -- (1); 
\draw[thick] (1) -- (3) -- (5) -- (2) -- (4) -- (1); 
  \end{tikzpicture}  
  \end{minipage}\hfill
    \begin{minipage}[b]{.49\linewidth}
\centering
  \begin{tikzpicture}[xscale=6,yscale=1.7]
\draw[thick] (0,0) -- (1,0);
\draw[thick] (0,.1) -- (0,.-.1);
\draw (.1,.05) -- (.1,-.05);
\draw (.2,.05) -- (.2,-.05);
\draw (.3,.05) -- (.3,-.05);
\draw (.4,.05) -- (.4,-.05);
\draw[thick] (.5,.1) -- (.5,.-.1);
\draw (.6,.05) -- (.6,-.05);
\draw (.7,.05) -- (.7,-.05);
\draw (.8,.05) -- (.8,-.05);
\draw (.9,.05) -- (.9,-.05);
\draw[thick] (1,.1) -- (1,.-.1);
\node at (0,-.4) {$0$};
\node at (1,-.4) {$1$};
\node at (.5,-.4) {$0.5$};
 \draw[fill=lightgray] (0,.2) rectangle (.6,.6);
 \draw[fill=lightgray] (.6,.7) rectangle (.9,1.1);
 \draw[fill=lightgray] (.9,1.2) rectangle (1,1.6);
 \draw[fill=lightgray] (0,1.2) rectangle (.2,1.6);
 \draw[fill=lightgray] (.2,1.7) rectangle (1,2.1);
 \draw[fill=lightgray] (0,1.7) rectangle (.1,2.1);
 \draw[fill=lightgray] (.1,2.2) rectangle (.5,2.6);
 \node at (-.1,.4) {$X_1$};
 \node at (-.1,.9) {$X_2$};
 \node at (-.1,1.4) {$X_3$};
 \node at (-.1,1.9) {$X_4$};
 \node at (-.1,2.4) {$X_5$};
  \end{tikzpicture}  
\end{minipage}
\caption{Illustration of the sets $X_i$ in Example~\ref{ex:convex_combination}.}
  \label{fig:convex_combination}
\end{figure}
  There are six vectors $\vect\xi$
  with non-empty $R_{\vect\xi}(X_1,\dots,X_5)$:
  \begin{align*}
    R_{(1,0,1,1,0)}(X_1,\dots,X_5) &= [0,0.1),\\
    R_{(1,0,1,0,1)}(X_1,\dots,X_5) &= [0.1,0.2),\\
    R_{(1,0,0,1,1)}(X_1,\dots,X_5) &= [0.2,0.5),\\
    R_{(1,0,0,1,0)}(X_1,\dots,X_5) &= [0.5,0.6),\\
    R_{(0,1,0,1,0)}(X_1,\dots,X_5) &= [0.6,0.9),\\
    R_{(0,0,1,1,0)}(X_1,\dots,X_5) &= [0.9,0.1).
  \end{align*}
So these sets $X_1,\dots,X_5$ correspond to the representation
\[
  \begin{pmatrix}
    0.6\\0.3\\0.3\\0.9\\0.4
  \end{pmatrix} = 0.1
 \begin{pmatrix}
    1\\0\\1\\1\\0
  \end{pmatrix} +0.1
 \begin{pmatrix}
    1\\0\\1\\0\\1
  \end{pmatrix} +0.3
 \begin{pmatrix}
    1\\0\\0\\1\\1
  \end{pmatrix} +0.1
 \begin{pmatrix}
    1\\0\\0\\1\\0
  \end{pmatrix} +0.3
 \begin{pmatrix}
    0\\1\\0\\1\\0
  \end{pmatrix} +0.1
 \begin{pmatrix}
    0\\0\\1\\1\\0
  \end{pmatrix},
\]
and extending this to the function value coordinate, we get
\begin{multline*}
  \vex[f](\vect x)\leqslant 0.1f(1,0,1,1,0)+0.1f(1,0,1,0,1)+0.3f(1,0,0,1,1)+0.1f(1,0,0,1,0)\\
  +0.3f(0,1,0,1,0)+0.1f(0,0,1,1,0)=0.3+0.3+0.9+0.1+0.3+0.1=2.
\end{multline*}
\end{example}
The following result, which was proved in~\cite{Luedtke12} (see also~\cite{Tawarmalani12}), is an
immediate consequence of Corollary~\ref{cor:set_interpretation}.
\begin{corollary}\label{cor:cav}
  If all edges are positive, then
  $\displaystyle\cav[f](\vec{x}) = \sum_{ij \in E} a_{ij} \min \{ x_i, x_j \}$.
\end{corollary}
\begin{proof}
  With $X_i=[0,x_i)$ for all $i\in[n]$, we get
  \[\cav[f](\vect x)\geq\sum_{ij\in E}a_{ij}\mu(X_i\cap X_j)=\sum_{ij\in
      E}a_{ij}\mu([0,\min\{x_i,x_j\})=\sum_{ij\in E}a_{ij}\min\{x_i,x_j\}.\]
  On the other hand, for any feasible choice of the sets $X_i$, we have, for all $i,j\in[n]$,
  \[\mu(X_i\cap X_j)\leq\min\{\mu(X_i),\,\mu(X_j)\}=\min\{x_i,x_j\}.\]
  With the assumption that $a_{ij}>0$ for all $ij\in E$, this implies
  \[\cav[f](\vect x)\leq \sum_{ij\in E}a_{ij}\min\{x_i,x_j\}.\qedhere\]
\end{proof}

As another consequence, we can combine convex hull characterizations of graphs of two bilinear
functions if they share at most one variable.
\begin{corollary}\label{cor:combination}
Let $f\colon[0,1]^k\to\reals$ and $g\colon[0,1]^{n-k+1}\to\reals$ be two bilinear functions given as
\[
f(\vect x) =\sum_{1\leq i<j\leq k}a_{ij}x_ix_j, \qquad g(\vect x) =\sum_{k\leq i<j\leq n}a_{ij}x_ix_j,
\] 
so that $f$ depends only on variables
  $x_1,\dotsc,x_k$, and $g$ depends only on variables $x_k,\dotsc,x_n$. Let
  $P,Q\subseteq[0,1]^{n(n+1)/2}$ be polytopes with $\pi[f](P)=X(f)$ and $\pi[g](Q)=X(g)$, such that
  $P$ is described by inequalities involving only the variables $x_1,\dots,x_k$ and $y_{ij}$
  with $1\leq i<j\leq k$, and $Q$ is described by inequalities involving only the variables $x_k,\dots,x_n$ and $y_{ij}$
  with $k\leq i<j\leq k$. Then $\pi[f+g](P\cap Q)=\X(f+g)$.
\end{corollary}
\begin{proof}
Fix $\vect x\in[0,1]^n$. By assumption and Corollary~\ref{cor:set_interpretation}, there are sets $X_1,\dots,X_n\in\mathcal L$ with
  $\mu(X_i)=x_i$ for all $i\in[k]$, and sets $X'_k,\dots,X'_n\in\mathcal L$ with $\mu(X'_i)=x_i$ for
  all $i\in[k,n]$, such that
\[
    \sum_{1\leq i<j\leq k}a_{ij}\mu(X_i\cap X_j) \eq \LB_P[f](\vect x), \quad \sum_{k\leq i<j\leq n}a_{ij}\mu(X'_i\cap X'_j) \eq \LB_Q[g](\vect x).
\]
  Applying a measure preserving bijection $[0,1)\to[0,1)$ that maps $\mathcal L$ to $\mathcal L$ and
  $X'_k$ to $X_k$, we can assume that $X_k=X'_k$, and then the sets
  $X_1,\dots,X_{k-1},X_k=X'_k,X'_{k+1},\dots,X'_n$ provide a certificate for $\vex[f+g](\vect
  x)=\LB_{P\cap Q}[f+g](\vect
  x)$:
  \[\sum_{1\leq i<j\leq k}a_{ij}\mu(X_i\cap X_j)+\sum_{k\leq i<j\leq n}a_{ij}\mu(X'_i\cap
    X'_j)=\LB_P[f](\vect x)+\LB_Q[g](\vect x)=\LB_{P\cap Q}[f+g](\vect x).\]
  The same argument works for $\cav[f+g](\vect x)=\UB_{P\cap Q}[f+g](\vect x)$.
  \end{proof}
  

\subsection{Alternative proof for cliques}\label{sec:cliques}
In order to illustrate the utility of the geometric characterization of $0$--$1$ polytopes in a simpler
setting than what we have for our main results, we start with an alternative proof for the following
result that was proved in~\cite{Rikun97,Sherali97a}.

\begin{theorem}\label{res:clique}
  If $G = K_n$, and all edge weights are equal to $1$, then $\X(f) = \pi[f](\P)$ where
  $\P\subseteq[0,1]^{n(n+1)/2}$ is the polytope described by the inequalities $y_{ij}\leq x_i$ and
  $y_{ij}\leq x_j$ for all $ij\in E$, together with
  \begin{equation}\label{eq:clique2}
  s x(V) - y(E) \leq s (s + 1)/2,\ s = 1, \dotsc, n-1.
  \end{equation}
\end{theorem}
To be precise, the result from~\cite{Rikun97,Sherali97a} is just one half of Theorem~\ref{res:clique}: in our
notation it says that
\[\vex[f](\vect x)=\max\left\{0,\,\max\left\{s x(V)-s(s+1)/2\,:\,s=1,2,\dots,n-1\right\}\right\}.\]
The correspondence between this statement in the original space $\reals^{n+1}$ and our version in
the extended space $\reals^{n(n+1)/2}$ comes from the fact that the only constraints enforcing lower
bounds on the $y$-variables are~\eqref{eq:clique2}, which put lower bounds on $y(E)$ which is
precisely the term corresponding to $f(\vect x)$ when the products $x_ix_j$ are replaced by the
variables $y_{ij}$. The other half of Theorem~\ref{res:clique} is an immediate consequence of
Corollary~\ref{cor:cav}: In order to describe $\cav[f](\vect x)$ it is sufficient to require
$y_{ij}\leq\min\{x_i,\,x_j\}$ for all $ij\in E$. In view of Corollaries~\ref{cor:set_interpretation}
and~\ref{cor:cav}, Theorem~\ref{res:clique} is a consequence of the following lemma.
\begin{lemma}\label{lem:vex_clique} For every $\vect x\in[0,1]^n$,
  \[\LB_P[f](\vect x)=\vex[f](\vect x)=s(x_1+\dotsb+x_n)-\binom{s+1}{2},\]
  where $s=\lfloor x_1+\dotsb+x_n\rfloor$.   
\end{lemma}
\begin{proof}
Since $\vex[f](\vect x)\geqslant\LB_P[f](\vect x)$ it is sufficient to show that 
\begin{align}
  \vex[f](\vect x)&\leqslant s(x_1+\dotsb+x_n)-\binom{s+1}{2},\label{eq:first_half}\\
  \LB_P[f](\vect x)&\geqslant s(x_1+\dotsb+x_n)-\binom{s+1}{2}.\label{eq:second_half}  
\end{align}
In order to show~\eqref{eq:first_half} using Corollary~\ref{cor:set_interpretation} we concatenate
intervals of lengths $x_1,\dotsc,x_n$ and obtain sets $X_i\subseteq[0,1)$ by interpreting the result
modulo $\ints$. More formally, the sets $X_i$ are defined as follows: for $i=1$, put
$X_1=[0,x_1)$. Now let $i\geqslant 2$, suppose $X_{i-1}$ has been defined already, set
$b=\sup X_{i-1}$, and put
\[X_{i}=
\begin{cases}
  [b,b+x_i) & \text{if }b+x_i\leqslant 1,\\
  [b,1)\cup[0,x_i-(1-b)) &\text{if }b+x_i>1.
\end{cases}
\]
\begin{figure}[htb]
  \begin{minipage}[b]{.49\linewidth}
\centering
  \begin{tikzpicture}[scale=2,every node/.style={circle,fill=black,draw,inner sep=2pt,outer
      sep=2pt}]
\node[label={right:$x_1=0.6$}] (1) at (0:1) {};
\node[label={[label distance=-.2cm]72:$x_2=0.3$}] (2) at (72:1) {};
\node[label={[label distance=-.2cm]144:$x_3=0.3$}] (3) at (144:1) {};
\node[label={[label distance=-.2cm]-144:$x_4=0.9$}] (4) at (-144:1) {};
\node[label={[label distance=-.2cm]-72:$x_5=0.4$}] (5) at (-72:1) {};
\draw[thick] (1) -- (2) -- (3) -- (4) -- (5) -- (1); 
\draw[thick] (1) -- (3) -- (5) -- (2) -- (4) -- (1); 
  \end{tikzpicture}  
  \end{minipage}\hfill
    \begin{minipage}[b]{.49\linewidth}
\centering
  \begin{tikzpicture}[xscale=6,yscale=1.7]
\draw[thick] (0,0) -- (1,0);
\draw[thick] (0,.1) -- (0,.-.1);
\draw (.1,.05) -- (.1,-.05);
\draw (.2,.05) -- (.2,-.05);
\draw (.3,.05) -- (.3,-.05);
\draw (.4,.05) -- (.4,-.05);
\draw[thick] (.5,.1) -- (.5,.-.1);
\draw (.6,.05) -- (.6,-.05);
\draw (.7,.05) -- (.7,-.05);
\draw (.8,.05) -- (.8,-.05);
\draw (.9,.05) -- (.9,-.05);
\draw[thick] (1,.1) -- (1,.-.1);
\node at (0,-.4) {$0$};
\node at (1,-.4) {$1$};
\node at (.5,-.4) {$0.5$};
\draw[fill=lightgray] (0,.2) rectangle (.6,.6); 
\draw[fill=lightgray] (.61,.2) rectangle (.9,.6);
\draw[draw=none,fill=lightgray] (.91,.2) rectangle (1,.6);
\draw (1,.2) -- (.91,.2) -- (.91,.6) -- (1,.6);
\draw[draw=none,fill=lightgray] (0,.7) rectangle (.2,1.1);
\draw (0,1.1) -- (.2,1.1) -- (.2,.7) -- (0,.7);
\draw[draw=none,fill=lightgray] (.21,.7) rectangle (1,1.1);
\draw (1,1.1) -- (.21,1.1) -- (.21,.7) -- (1,.7);
\draw[draw=none,fill=lightgray] (0,1.2) rectangle (.1,1.6);
\draw (0,1.6) -- (.1,1.6) -- (.1,1.2) -- (0,1.2);
\draw[fill=lightgray] (.11,1.2) rectangle (.5,1.6);
\node at (.3,.4) {$X_1$};
\node at (.75,.4) {$X_2$};
\node at (.95,.4) {$X_3$};
\node at (.1,.9) {$X_3$};
\node at (.55,.9) {$X_4$};
\node at (.05,1.4) {$X_4$};
\node at (.3,1.4) {$X_5$};
\draw [decorate,decoration={brace,amplitude=10pt},xshift=0pt,yshift=0pt]
(0,1.7) -- (0.5,1.7) node [black,midway,yshift=.6cm] {$B$};
\draw [decorate,decoration={brace,amplitude=10pt},xshift=0pt,yshift=0pt]
(.5,1.7) -- (1,1.7) node [black,midway,yshift=.6cm] {$A$};
  \end{tikzpicture}  
\end{minipage}
\caption{Illustration of the construction of the sets $X_i$ in the proof of Lemma~\ref{lem:vex_clique}.}
  \label{fig:clique_sets}
\end{figure}
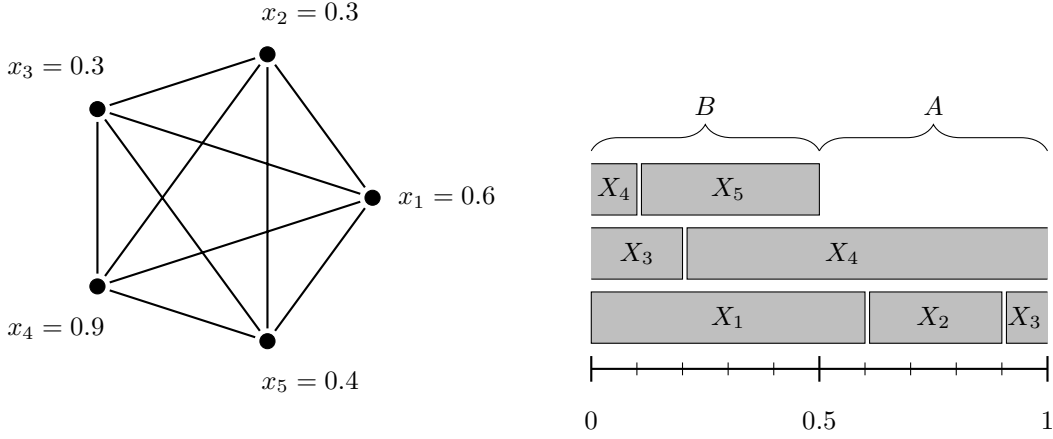
This construction is illustrated in Figure~\ref{fig:clique_sets} where $s=\lfloor 2.5\rfloor=2$. Now
$[0,1)=A\cup B$, where
\begin{align*}
  A&=\{t\in[0,1)\ :\ \lvert\{i\in[n]\ :\ t\in X_i\}\rvert=s\},& B&=\{t\in[0,1)\ :\ \lvert\{i\in[n]\ :\ t\in
  X_i\}\rvert=s+1\},
\end{align*}
and $\mu( B)=(x_1+\dotsb+x_n)-s$, $\mu( A)=s+1-(x_1+\dotsb+x_n)$. Therefore,
\begin{align*}
\sum_{1\leqslant i<j\leqslant n}\mu( X_i\cap X_j) &= \binom{s}{2}\mu(
A)+\binom{s+1}{2}\mu( B) \\
&= \binom{s}{2}\left[s+1-(x_1+\dotsb+x_n)\right]+\binom{s+1}{2}\left[(x_1+\dotsb+x_n)-s\right]\\
&= (x_1+\dotsb+x_n)\left[\binom{s+1}{2}-\binom{s}{2}\right]+(s+1)\binom{s}{2}-s\binom{s+1}{2}\\
&= s(x_1+\dotsb+x_n)-\binom{s+1}{2},
\end{align*}
and this implies~\eqref{eq:first_half}. If $s<n-1$ then~\eqref{eq:second_half} is obvious: for $s=0$
the right-hand side is zero, and for $1\leqslant s\leqslant n-1$,
\[\sum_{1\leqslant i<j\leqslant n}y_{ij}\geqslant s(x_1+\dotsb+x_n)-\binom{s+1}{2}\]
is one of the clique inequalities. Finally, if $s=n$ then
\[\sum_{1\leqslant i<j\leqslant n}y_{ij}=\binom{n}{2}=n^2-\binom{n+1}{2}=n(x_1+\dotsb+x_n)-\binom{n+1}{2}.\qedhere\]
\end{proof}

\section{Proofs of Main Results}
\label{sec:mainproofs}
In Sections~\ref{sec:clique_minus} and~\ref{sec:cycles} we will use
Corollary~\ref{cor:set_interpretation} to prove Theorems~\ref{thm:clique_minus}
and~\ref{thm:cycles}.

\subsection{Proof of Theorem~\ref{thm:clique_minus}}\label{sec:clique_minus}

We will establish this result by proving $\vex[f](\vect x)\leqslant \LB_P[f](\vect x)$. Without loss
of generality we assume $x_n\leqslant x_{n-1}$ and
$x_1\geqslant x_2\geqslant\dotsb\geqslant x_{n-2}$, and we proceed as follows. We start with
$X_n=[0,x_n)$ and $X_{n-1}=[0,x_{n-1})$ and construct the sets $X_1,\dotsc,X_{n-2}$ as described in
Algorithm~\ref{alg:construction}.
\begin{algorithm}
  \caption{Construction of the sets $X_i$ in the proof of
    Theorem~\ref{thm:clique_minus}}\label{alg:construction}
  \begin{tabbing}
    .....\=.....\=.....\=............................. \kill \\[-3ex]
    $(a,b)\leftarrow (x_n,x_{n-1})$\\
    \textbf{for} $k=1,2,\dotsc,n-2$ \textbf{do}\\
    \> \textbf{if} $x_k\leqslant 1-b$ \textbf{then}\\
    \> \> $X_k\leftarrow[b,b+x_k)$\\
    \> \> $b\leftarrow b+x_k$\\
    \> \textbf{else if} $x_k\leqslant 1-a$ \textbf{then}\\
    \> \> $X_k\leftarrow[b,1)\cup[a,a+x_k+b-1)$\\
    \> \> $(a,b)\leftarrow(a+x_k+b-1,1)$\\
    \> \textbf{else}\\
    \> \> $X_k\leftarrow[a,1)\cup[0,x_k+a-1)$\\
    \> \> $a\leftarrow x_k+a-1$
  \end{tabbing}
\end{algorithm}
\begin{example}
  For $n=6$ two different outcomes of Algorithm~\ref{alg:construction} are illustrated in
  Figure~\ref{fig:ex_2}: For $\vect x=(0.9,\,0.6,\,0.2,\,0.1,\,0.6,\,0.4)$ the algorithm terminates
  with $b=1$, while for $\vect x=(0.9,\,0.8,\,0.2,\,0.1,\,0.6,\,0.4)$, we have still $b<1$ in the
  end.
  \begin{figure}[htb]
    \centering
    \begin{minipage}[b]{.49\linewidth}
      \centering
      \begin{tikzpicture}[xscale=6,yscale=1.2]
\draw[thick] (0,0) -- (1,0);
\draw[thick] (0,.1) -- (0,.-.1);
\draw (.1,.05) -- (.1,-.05);
\draw (.2,.05) -- (.2,-.05);
\draw (.3,.05) -- (.3,-.05);
\draw (.4,.05) -- (.4,-.05);
\draw[thick] (.5,.1) -- (.5,.-.1);
\draw (.6,.05) -- (.6,-.05);
\draw (.7,.05) -- (.7,-.05);
\draw (.8,.05) -- (.8,-.05);
\draw (.9,.05) -- (.9,-.05);
\draw[thick] (1,.1) -- (1,.-.1);
\node at (0,-.3) {$0$};
\node at (1,-.3) {$1(=b)$};
\node at (.5,-.3) {$0.5$};
\node at (.8,-.3) {$a$};
\draw[dashed] (.8,0) -- (.8,2.2);
\draw[dashed] (1,0) -- (1,2.2);
 \draw[fill=lightgray] (0,.2) rectangle (.6,.6);
 \draw[fill=lightgray] (0,.7) rectangle (.4,1.1);
 \draw[fill=lightgray] (.61,.2) rectangle (1,.6);
 \draw[fill=lightgray] (.41,.7) rectangle (.6,1.1);
 \draw[fill=lightgray] (0,1.2) rectangle (.3,1.6);
 \draw[fill=lightgray] (.61,.7) rectangle (1,1.1);
 \draw[fill=lightgray] (.31,1.2) rectangle (.5,1.6);
 \draw[fill=lightgray] (.51,1.2) rectangle (.7,1.6);
 \draw[fill=lightgray] (.71,1.2) rectangle (.8,1.6);
 \node at (.3,.4) {$X_5$};
 \node at (.8,.4) {$X_1$};
 \node at (.2,.9) {$X_6$};
 \node at (.5,.9) {$X_1$};
 \node at (.8,.9) {$X_2$};
 \node at (.15,1.4) {$X_1$};
 \node at (.4,1.4) {$X_2$};
 \node at (.6,1.4) {$X_3$};
 \node at (.75,1.4) {$X_4$};
 \node at (.5,-.7) {$\vect x=(0.9,\,0.6,\,0.2,\,0.1,\,0.6,\,0.4)$};
\end{tikzpicture}
\end{minipage}\hfill
\begin{minipage}[b]{.49\linewidth}
  \centering
  \begin{tikzpicture}[xscale=6,yscale=1.2]
\draw[thick] (0,0) -- (1,0);
\draw[thick] (0,.1) -- (0,.-.1);
\draw (.1,.05) -- (.1,-.05);
\draw (.2,.05) -- (.2,-.05);
\draw (.3,.05) -- (.3,-.05);
\draw (.4,.05) -- (.4,-.05);
\draw[thick] (.5,.1) -- (.5,.-.1);
\draw (.6,.05) -- (.6,-.05);
\draw (.7,.05) -- (.7,-.05);
\draw (.8,.05) -- (.8,-.05);
\draw (.9,.05) -- (.9,-.05);
\draw[thick] (1,.1) -- (1,.-.1);
\node at (0,-.3) {$0$};
\node at (1,-.3) {$1$};
\node at (.5,-.3) {$0.5$};
\node at (.1,-.3) {$a$};
\node at (.9,-.3) {$b$};
 \draw[fill=lightgray] (0,.2) rectangle (.6,.6);
 \draw[fill=lightgray] (0,.7) rectangle (.4,1.1);
 \draw[fill=lightgray] (.61,.2) rectangle (1,.6);
 \draw[fill=lightgray] (.41,.7) rectangle (.6,1.1);
 \draw[fill=lightgray] (0,1.2) rectangle (.3,1.6);
 \draw[fill=lightgray] (.61,.7) rectangle (1,1.1);
 \draw[fill=lightgray] (.31,1.2) rectangle (.6,1.6);
 \draw[fill=lightgray] (.61,1.2) rectangle (.8,1.6);
 \draw[fill=lightgray] (.81,1.2) rectangle (.9,1.6);
 \draw[fill=lightgray] (0,1.7) rectangle (.1,2.1);
 \node at (.3,.4) {$X_5$};
 \node at (.8,.4) {$X_1$};
 \node at (.2,.9) {$X_6$};
 \node at (.5,.9) {$X_1$};
 \node at (.8,.9) {$X_2$};
 \node at (.15,1.4) {$X_1$};
 \node at (.45,1.4) {$X_2$};
 \node at (.7,1.4) {$X_3$};
 \node at (.85,1.4) {$X_4$};
 \node at (.05,1.9) {$X_2$};
 \node at (.5,-.7) {$\vect x=(0.9,\,0.8,\,0.2,\,0.1,\,0.6,\,0.4)$};
 \draw[dashed] (.1,0) -- (.1,2.2);
 \draw[dashed] (.9,0) -- (.9,2.2);
\end{tikzpicture}
\end{minipage}
\caption{The sets $X_i$ constructed by Algorithm~\ref{alg:construction} for two vectors
  $\vect x$.}\label{fig:ex_2}
\end{figure}
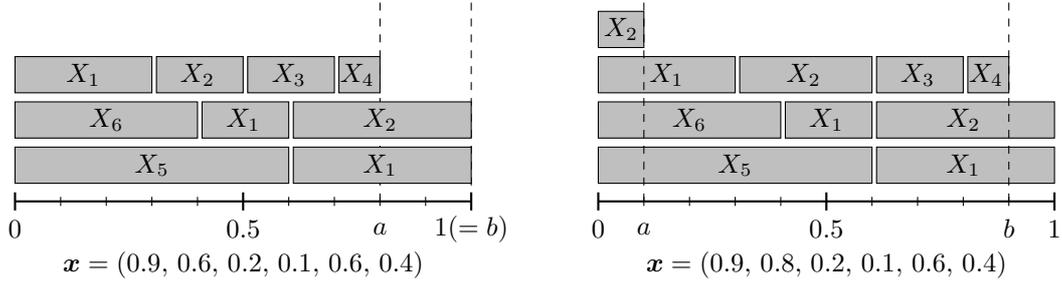
\end{example}
Setting $y_{ij}=\mu( X_i\cap X_j)$ for all $ij\in E$, Corollary~\ref{cor:set_interpretation}
implies $\vex[f](\vect x)\leqslant \sum_{ij\in E}y_{ij}$, and it is sufficient to show that
$\vect y$ is an optimal solution for the LP defining $\LB_P[f](\vect x)$:
\[\text{Minimize } \sum_{ij\in E}y_{ij} \text{ subject to the McCormick inequalities~\eqref{eq:mccormick} and \eqref{eq:clique_minus_1}--\eqref{eq:clique_minus_3}}.\]
Our argument will be based on expressing $\sum_{ij\in E}y_{ij}$ in terms of the variables $x_i$ and
then arguing the inequalities listed in Theorem~\ref{thm:clique_minus} imply that this expression is a
lower bound for $\LB_P[f](\vect x)$. In other words, we need to verify
$\sum_{ij\in E}y_{ij}\leq\sum_{ij\in E}y'_{ij}$ for every $\vect y'$ with $(\vect x,\vect y')\in
P$. The following observations turn out to be useful.
\begin{itemize}
  \item During the runtime of the algorithm the parameter $b$ is never decreasing.
\item After every step of Algorithm~\ref{alg:construction} there is an integer $s$ such that
  \[\lvert\{i\in[n]\,:\,t\in X_i\}\rvert=
    \begin{cases}
      s+2 & \text{for }t<a,\\
      s+1 & \text{for }a\leq t<b,\\
      s & \text{for }b\leq t<1
    \end{cases}
  \]
  if $b<1$, and
    \[\lvert\{i\in[n]\,:\,t\in X_i\}\rvert=
    \begin{cases}
      s+1 & \text{for }t<a,\\
      s & \text{for }a\leq t<1,\\
     \end{cases}
   \]
   if $b=1$.
   \item If $b=1$ at termination of the algorithm, and $s$ is the integer with $\lvert\{i\in[n]\,:\,t\in
      X_i\}\rvert=s$ for $t\in[a,1)$ then $a=x_1+x_2+\dotsb+x_n-s$.
    \item If $b<1$ at termination of the algorithm, and $s$ is the integer with $\lvert\{i\in[n]\,:\,t\in
      X_i\}\rvert=s$ for $t\in[b,1)$ then $a=x_1+x_2+\dotsb+x_s+x_n-s$ and
      $b=x_{s+1}+x_{s+2}+\dotsb+x_{n-1}$.    
\end{itemize}
\begin{description}
\item[Case 1] $b=1$. Then $1\leq s\leq n-1$. As in the proof of Lemma~\ref{lem:vex_clique},
  \[\sum_{ij\in E}y_{ij} = \sum_{1\leq i<j\leq n}y_{ij} -
    y_{n-1,n}=s(x_1+\dotsb+x_{n})-x_n-\binom{s+1}{2}.\]
  For $s\leq n-2$, we use~\eqref{eq:clique_minus_3}:
  \[s(x_1+\dots+x_{n})-x_n-\binom{s+1}{2}\leq s x(V)
    -y'_{n,n-1}-\binom{s+1}{2}\stackrel{\eqref{eq:clique_minus_3}}{\leq}y'_{ij}(E).\]
  For $s=n-1$, we combine~\eqref{eq:clique_minus_0} and~\eqref{eq:clique_minus_1}:
  \begin{multline*}
    y'(E)=\left[y'(E(V\setminus\{n-1,n\}))
      +\frac12\sum_{i=1}^{n-2}\left(y'_{i,n-1}+y'_{in}\right)\right]+\frac12\sum_{i=1}^{n-2}\left(y'_{i,n-1}+y'_{in}\right)\\
    \stackrel{\eqref{eq:clique_minus_0},\eqref{eq:clique_minus_1}}{\geq}(n-2)\left[x(V\setminus\{n-1,n\})+\frac{x_{n-1}+x_n}{2}\right]-\binom{n-1}{2}\\
    +x\left(V\setminus\{n-1,n\}\right)+\frac{(n-2)(x_{n-1}+x_n)}{2}-(n-2)\\
    =(n-1)x(V)-(x_{n-1}+x_n)-\binom{n}{2}+1\geq(n-1)x(V)-x_n-\binom{n}{2},    
  \end{multline*}
  as required.
\item[Case 2] $b<1$. Then $0\leq s\leq n-2$, and
  \begin{align*}
    \sum_{ij\in E}y_{ij} &=
                           a\left[\binom{s+2}{2}-1\right]+(x_n-a)\left[\binom{s+1}{2}-1\right]+(b-x_n)\binom{s+1}{2}+(1-b)\binom{s}{2}\\
                         &=a(s+1)-x_n+bs+\binom{s}{2}\\
                         &= (x_1+\dotsb+x_s+x_n-s)(s+1)-x_n+(x_{s+1}+\dotsb+x_{n-1})s+\binom{s}{2}\\
                         &=sx(V)+(x_1+\dots+x_s)-s-\binom{s+1}{2}\\
                         &=s\left[x\left(V\setminus\{n-1,n\}\right)+\frac{x_{n-1}+x_n}{2}\right]+\frac12\sum_{i=1}^{s}\left(2x_i+x_{n-1}+x_n-2\right)-\binom{s+1}{2}.   
  \end{align*}
  Next we verify that this is a lower bound for $\LB_P[f](\vect x)$. We start with the inequality
  \begin{multline*}
    \sum_{ij\in E}y'_{ij}=\sum_{ij\in
      E(V\setminus\{n-1,n\})}y'_{ij}+\sum_{i=1}^{n-2}\left(y'_{i,n-1}+y'_{in}\right)\\
    \geq\left[\sum_{ij\in
      E(V\setminus\{n-1,n\})}y'_{ij}+\frac12\sum_{i=1}^{n-2}\left(y'_{i,n-1}+y'_{in}\right)\right]+\frac12\sum_{i=1}^s\left(y'_{i,n-1}+y'_{in}\right).    
  \end{multline*}
  We use \eqref{eq:clique_minus_0} and~\eqref{eq:clique_minus_1} to bound the second and the first
  part, respectively:
  \begin{align*}
    \sum_{ij\in
    E(V\setminus\{n-1,n\})}y'_{ij}+\frac12\sum_{i=1}^{n-2}\left(y'_{i,n-1}+y'_{in}\right)&\stackrel{\eqref{eq:clique_minus_1}}{\geq}s\left[x\left(V\setminus\{n-1,n\}\right)+\frac{x_{n-1}+x_n}{2}\right]-\binom{s+1}{2},\\
    \frac12\sum_{i=1}^s\left(y'_{i,n-1}+y'_{in}\right) &\stackrel{\eqref{eq:clique_minus_0}}{\geq} \frac12\sum_{i=1}^s\left(2x_i+x_{n-1}+x_n-2\right),
  \end{align*}
  and as a consequence $y'(E)\geq y(E)$, as required.
\end{description}
The second part of the theorem (that $P$ is a minimal extension of $X(f)$) is proved by identifying,
for each inequality listed in Theorem~\ref{thm:clique_minus}, a point $(\vect x,\vect y)$ which is
contained in the polytope $P'$ obtained from $P$ by omitting the inequality, such that $\pi[f](\vect
x,\vect y)\not\in X(f)$. This is described in detail in Appendix~\ref{sec:minimality_proof}. 

\subsection{Proof of Theorem~\ref{thm:cycles}}\label{sec:cycles}
Throughout all indices are in $[n]$ and have to be read modulo $n$ in the obvious way. In particular,  $n \equiv 0$ and $n+1\equiv 1$. The $n$-cycle corresponds to the function
\[
f(\vect x)\eq\sum_{i=1}^{n}a_ix_ix_{i+1},
\] 
where $\vect a$ is arbitrary with $a_{i}\neq 0$ for all $i\in[n]$. Let $P\subseteq[0,1]^{2n}$ be the polytope described by the McCormick inequalities
together with~\eqref{eq:cycle_1} and~\eqref{eq:cycle_2}. We claim that $\pi[f](P)=X(f)$, and
that~\eqref{eq:cycle_1} (resp.~\eqref{eq:cycle_2}) can be omitted if $\lvert E^-\rvert$
(resp. $\lvert E^+\rvert$) is even.

We need to show that for every $\vect x\in[0,1]^n$ we have $\vex[f](\vect x)=\LB_P[f](\vect x)$ and
$\cav[f](\vect x)=\UB_P[f](\vect x)$. We present the argument for $\cav[f](\vect x)=\UB_P[f](\vect
x)$ in detail, as $\vex[f](\vect x)=\LB_P[f](\vect x)$ can be proved similarly. Fix $\vect x\in[0,1]^n$ and put
\begin{align*}
\mu_i&=\min\{x_i,x_{i+1}\}, \\ 
\eta_i&=\max\{0,x_i+x_{i+1}-1\}, \\ 
A&=x(V^+)-x(V^-)+\left\lfloor\frac{\lvert E^-\rvert}{2}\right\rfloor.
\end{align*}
Then
\begin{align}
  \UB_P[f](\vect x) &= \max\left\{\sum_{i=1}^na_iy_i\ :\ \eta_i\leqslant y_i\leqslant\mu_i,\
    y(E^+)-y(E^-)\leqslant A\right\}\label{eq:UB_primal}\\
&\nonumber = \min\Bigg\{\sum_{i=1}^n\left(\mu_iz_i-\eta_iw_i\right)+A\alpha\ :\ 
z_i-w_i+\alpha\geqslant a_i\text{ for }i\in E^+,\\ 
&\qquad\qquad\qquad\qquad\qquad\qquad\qquad z_i-w_i-\alpha\geqslant a_i\text{ for }i\in
E^-,\ z_i,w_i,\alpha\geqslant 0\Bigg\}. \label{eq:UB_dual}
\end{align}
W.l.o.g. we assume $\lvert a_n\rvert\leqslant \lvert a_i\rvert$ for all $i\in[n]$, and we define sets
$X_i\subseteq[0,1)$ as follows. Let $X_1=[0,x_1)$ and
\[X_2=
\begin{cases}
[0,x_2) & \text{if } a_1>0,\\
[1-x_2,1) & \text{if } a_1<0.
\end{cases}\]
For $i\geqslant 3$ we set $E^-_i=\{j\ :\ i\leqslant j\leqslant n,\ a_j<0\}$ and define $X_i$
depending on the parity of $\lvert E_i^-\rvert$ and the sign of $a_{i-1}$. Intuitively, we can think
of filling a bucket of capacity $x_i$ from the reservoirs $R_1=X_{i-1}\setminus
X_1$, $R_2=X_{i-1}\cap X_1$, $R_3=[0,1)\setminus(X_1\cup X_{i-1})$ and $R_4=X_1\setminus X_{i-1}$, and there are
two objectives:
\begin{enumerate}
\item If $a_{i-1}>0$ we want to maximize $\mu( X_{i-1}\cap X_i)$, so the reservoirs
  $R_1$ and $R_2$ are used before $R_3$ and $R_4$. If $a_{i-1}<0$ then we want to minimize $\mu( X_{i-1}\cap X_i)$, so $R_3$ and $R_4$ are used before $R_1$ and $R_2$.
\item If $\lvert E_i^-\rvert$ is odd we want to minimize $\mu( X_{1}\cap X_i)$, so $R_1$ is
  used before $R_2$ and $R_3$ before $R_4$. If $\lvert E_i^-\rvert$ is even we want to maximize $\mu( X_{1}\cap X_i)$, so $R_2$ is
  used before $R_1$ and $R_4$ before $R_3$.
\end{enumerate}
\begin{algorithm}
  \caption{\texttt{Bucket}$(Y_1,Y_2,Y_3,Y_4,x)$}\label{alg:bucket}
  \begin{tabbing}
    .....\=.....\=.....\=...............................\kill\\[-3ex]
    \textbf{Input:} \>\>\> A partition
    $[0,1)=Y_1\cup Y_2\cup Y_3\cup Y_4$ (with $Y_k\in\mathcal L$) and capacity $x\in[0,1]$\\[1ex]
    Initialize $X\leftarrow\emptyset$ and $k\leftarrow 1$\\
    \textbf{while } $\mu( X)<x$ \textbf{ do}\\ 
    \> Let $Z\in\mathcal L$ be a subset of $Y_k$ with $\mu( Z)=\min\left\{\mu( Y_k),\,x-\mu( X)\right\}$\\
    \> $X\leftarrow X\cup Z$\\
    \> $k\leftarrow k+1$\\[1ex]
    \textbf{Output:} $X\in\mathcal L$ with $\mu( X)=x$
  \end{tabbing}
\end{algorithm}
More formally,
\[X_i=
\begin{cases}
  \texttt{Bucket}(R_1,R_2,R_3,R_4,x_i) & \text{if }a_{i-1}>0\text{ and }\lvert E^-_i\rvert\text{
    odd},\\
  \texttt{Bucket}(R_2,R_1,R_4,R_3,x_i) & \text{if }a_{i-1}>0\text{ and }\lvert E^-_i\rvert\text{
    even},\\
  \texttt{Bucket}(R_3,R_4,R_1,R_2,x_i) & \text{if }a_{i-1}<0\text{ and }\lvert E^-_i\rvert\text{
    odd},\\
  \texttt{Bucket}(R_4,R_3,R_2,R_1,x_i) & \text{if }a_{i-1}<0\text{ and }\lvert E^-_i\rvert\text{
    even}.
\end{cases}
\]
where the function \texttt{Bucket} is described in Algorithm~\ref{alg:bucket}. Note that this
corresponds to a solution with 
\[y_i=\mu( X_i\cap X_{i+1})=
\begin{cases}
  \mu_i & \text{for }i\in E^+\setminus\{n\},\\
  \eta_i & \text{for }i\in E^-\setminus\{n\}.
\end{cases}
\]
\begin{example}
  \begin{figure}[htb]
    \begin{minipage}[b]{.49\linewidth}
\centering
  \begin{tikzpicture}[scale=2,every node/.style={circle,fill=black,draw,inner sep=2pt,outer
      sep=2pt}]
\node[label={right:$x_1=0.6$}] (1) at (0:1) {};
\node[label={[label distance=-.2cm]45:$x_2=0.5$}] (2) at (45:1) {};
\node[label={[label distance=-.4cm]90:$x_3=0.3$}] (3) at (90:1) {};
\node[label={[label distance=-.2cm]135:$x_4=0.5$}] (4) at (135:1) {};
\node[label={[label distance=0cm]180:$x_5=0.4$}] (5) at (180:1) {};
\node[label={[label distance=-.2cm]225:$x_6=0.6$}] (6) at (225:1) {};
\node[label={[label distance=-.4cm]270:$x_7=0.5$}] (7) at (270:1) {};
\node[label={[label distance=-.2cm]315:$x_8=0.6$}] (8) at (315:1) {};
\draw[thick] (1) to node[draw=none,fill=none,left] {$+$} (2) to node[draw=none,fill=none,below]
{$-$} (3) to node[draw=none,fill=none,below] {$+$} (4) to node[draw=none,fill=none,right] {$-$} (5)
to node[draw=none,fill=none,right] {$-$} (6) to node[draw=none,fill=none,above] {$+$} (7) to
node[draw=none,fill=none,above] {$+$} (8) to node[draw=none,fill=none,left] {$+$} (1); 
  \end{tikzpicture}\hfill
\end{minipage}
\begin{minipage}[b]{.49\linewidth}
\centering
  \begin{tikzpicture}[xscale=6,yscale=1.2]
\draw[thick] (0,0) -- (1,0);
\draw[thick] (0,.1) -- (0,.-.1);
\draw (.1,.05) -- (.1,-.05);
\draw (.2,.05) -- (.2,-.05);
\draw (.3,.05) -- (.3,-.05);
\draw (.4,.05) -- (.4,-.05);
\draw[thick] (.5,.1) -- (.5,.-.1);
\draw (.6,.05) -- (.6,-.05);
\draw (.7,.05) -- (.7,-.05);
\draw (.8,.05) -- (.8,-.05);
\draw (.9,.05) -- (.9,-.05);
\draw[thick] (1,.1) -- (1,.-.1);
\node at (0,-.4) {$0$};
\node at (1,-.4) {$1$};
\node at (.5,-.4) {$0.5$};
 \draw[fill=lightgray] (0,.2) rectangle (.6,.6);
 \draw[fill=lightgray] (0,.7) rectangle (.5,1.1);
 \draw[fill=lightgray] (.5,1.2) rectangle (.8,1.6);
 \draw[fill=lightgray] (.3,1.7) rectangle (.8,2.1);
 \draw[fill=lightgray] (.8,2.2) rectangle (1,2.6); 
 \draw[fill=lightgray] (.1,2.2) rectangle (.3,2.6);
 \draw[fill=lightgray] (0,2.7) rectangle (.1,3.1);
 \draw[fill=lightgray] (.3,2.7) rectangle (.8,3.1);
 \draw[fill=lightgray] (0,3.2) rectangle (.1,3.6);
 \draw[fill=lightgray] (.3,3.2) rectangle (.7,3.6);
 \draw[fill=lightgray] (0,3.7) rectangle (.1,4.1);
 \draw[fill=lightgray] (.2,3.7) rectangle (.7,4.1);
 \node at (-.1,.4) {$X_1$};
 \node at (-.1,.9) {$X_2$};
 \node at (-.1,1.4) {$X_3$};
 \node at (-.1,1.9) {$X_4$};
 \node at (-.1,2.4) {$X_5$};  
 \node at (-.1,2.9) {$X_6$};
 \node at (-.1,3.4) {$X_7$};
 \node at (-.1,3.9) {$X_8$};
\end{tikzpicture}  
    \end{minipage}
    \caption{Constructing the sets $X_i$, where the edge labels indicate the sign of
      the coefficient $a_i$.}
    \label{fig:cycle}
  \end{figure}
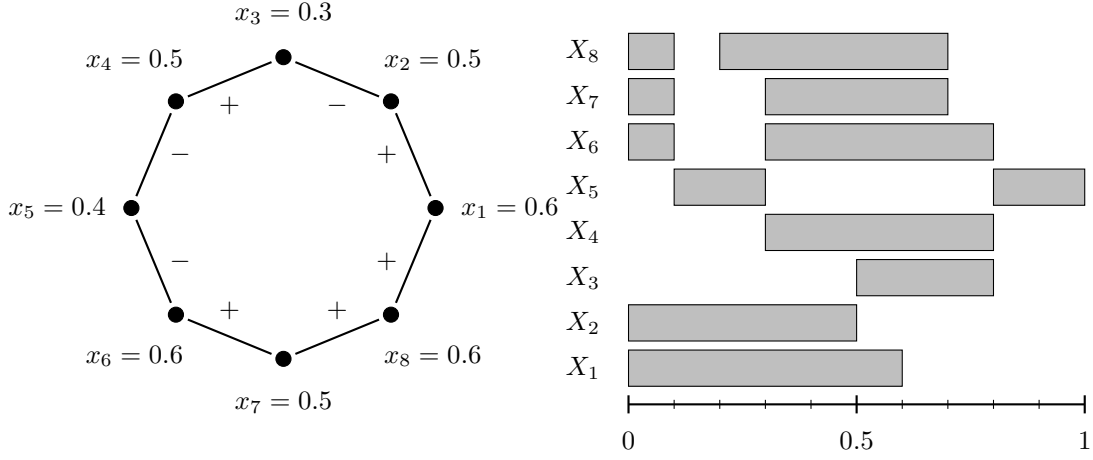
  The construction is illustrated in Figure~\ref{fig:cycle}, and with $\lvert a_i\rvert=1$ for all
  $i\in[8]$ the corresponding objective value is $2.3$ and we conclude $\cav[f](\vect x)=\UB_P[f](\vect x)=2.3$
  because~\eqref{eq:cycle_1} becomes
\[x(V^-)-x(V^+)+y(E^+)-y(E^-)=0.4-1.7+f(\vect x)\leqslant \left\lfloor\frac{\lvert E^-\rvert}{2}\right\rfloor=1.\]
\end{example}
Next we define the \emph{defects}
\[\delta_i=
\begin{cases}
  \mu( X_1\cap X_i)-\max\{0,x_1+x_i-1\} & \text{if }\lvert E^-_i\rvert\text{ is odd},\\
\min\{x_i,x_1\}-\mu( X_1\cap X_i)& \text{if }\lvert E^-_i\rvert\text{ is even}.
\end{cases}
\]
We have $\delta_i\geqslant 0$ for all $i\in[n]$ and
\[\cav[f](\vect x)\geqslant\sum_{i=1}^na_{i}\mu( X_i\cap X_{i+1})=\sum_{i\in E^+}a_i\mu_i+\sum_{i\in
  E^-}a_i\eta_i-\lvert a_n\rvert\delta_n,\]
so in order to complete the proof of the claim $\cav[f](\vect x)=\UB_P[f](\vect x)$ it is sufficient to
prove
\begin{equation} \label{eq:target}
\UB_P[f](\vect x)\leqslant \sum_{i\in E^+}a_i\mu_i+\sum_{i\in E^-}a_i\eta_i-\lvert
a_n\rvert\delta_n. 
\end{equation}
If $\delta_n=0$ then this follows immediately from the McCormick inequalities which imply
\[a_iy_i\leqslant
\begin{cases}
  a_i\mu_i & \text{for }i\in E^+,\\
a_i\eta_i & \text{for }i\in E^-.
\end{cases}
\]
For $\delta_n>0$ the claim is a consequence of the following lemma.
\begin{lemma}\label{lem:cycle_UB}
If $\delta_n>0$ then
$\displaystyle\delta_n=\sum_{i\in E^+}\mu_i-\sum_{i\in E^-}\eta_i-A$.
\end{lemma}
Using Lemma~\ref{lem:cycle_UB}, we can prove inequality~\eqref{eq:target} by LP duality. If
$a_n>0$ we define a solution for the dual problem~\eqref{eq:UB_dual} by $\alpha=a_n$ and (using the
assumption that $\lvert a_n\rvert\leqslant\lvert a_i\rvert$ for all $i\in[n]$)
\begin{align*}
  z_i&=
  \begin{cases}
    a_i-a_n &\text{for }i\in E^+,\\
     0 &\text{for }i\in E^-,
  \end{cases} &
 w_i&=
  \begin{cases}
     0&\text{for }i\in E^+,\\
     -a_i-a_n &\text{for }i\in E^-.
  \end{cases}
\end{align*}
This is a feasible solution for~\eqref{eq:UB_dual} with objective value
\begin{multline*}
  \sum_{i\in E^+}\mu_i(a_i-a_n)+\sum_{i\in E^-}\eta_i(a_i+a_n)+Aa_n=\sum_{i\in
    E^+}a_i\mu_i+\sum_{i\in E^-}a_i\eta_i+a_n\left(A-\sum_{i\in E^+}\mu_i+\sum_{i\in
      E^-}\eta_i\right)\\
=\sum_{i\in E^+}a_i\mu_i+\sum_{i\in E^-}a_i\eta_i-a_n\delta_n.
\end{multline*}
Similarly, for $a_n<0$ we define a solution for the dual problem~\eqref{eq:UB_dual} by
$\alpha=-a_n$ and 
\begin{align*}
  z_i&=
  \begin{cases}
    a_i+a_n &\text{for }i\in E^+,\\
     0 &\text{for }i\in E^-,
  \end{cases} &
 w_i&=
  \begin{cases}
     0&\text{for }i\in E^+,\\
     a_n-a_i &\text{for }i\in E^-.
  \end{cases}
\end{align*}
with objective value
\begin{multline*}
  \sum_{i\in E^+}\mu_i(a_i+a_n)+\sum_{i\in E^-}\eta_i(a_i-a_n)-Aa_n=\sum_{i\in
    E^+}a_i\mu_i+\sum_{i\in E^-}a_i\eta_i+a_n\left(\sum_{i\in E^+}\mu_i-\sum_{i\in
      E^-}\eta_i-A\right)\\
=\sum_{i\in E^+}a_i\mu_i+\sum_{i\in E^-}a_i\eta_i+a_n\delta_n.
\end{multline*}
Before proving Lemma~\ref{lem:cycle_UB} we show that the sequence $(\delta_i)_{i=2,\dotsc,n}$ is decreasing, and
therefore in the proof of Lemma~\ref{lem:cycle_UB} we may assume $\delta_i>0$ for all $i\in[n]$.
\begin{lemma}\label{lem:mono}
  For all $i\in\{3,\dotsc,n\}$ we have $\delta_i\leqslant\delta_{i-1}$.
\end{lemma}
\begin{proof}
  We check the following implications.
  \begin{enumerate}
  \item $\lvert E_{i-1}\rvert$ odd, $a_{i-1}>0$ $\implies$
    $\delta_i=\delta_{i-1}-
    \begin{cases}
\max\{0,\,x_{i-1}-x_i,\,(x_1+x_i-1)\} & \text{if }x_1+x_{i-1}\leqslant 1,\\
\max\{0,\,1-x_1-x_i,\,x_i-x_{i-1}\}  & \text{if }x_1+x_{i-1}\geqslant 1.    
    \end{cases}$
  \item $\lvert E_{i-1}\rvert$ odd, $a_{i-1}<0$ $\implies$ $\delta_i=\delta_{i-1}-
    \begin{cases}
\max\{0,\,x_{1}-x_i,\,(x_{i-1}+x_i-1)\} & \text{if }x_1+x_{i-1}\leqslant 1,\\
\max\{0,\,1-x_{i-1}-x_i,\,x_i-x_{1}\}  & \text{if }x_1+x_{i-1}\geqslant 1.    
    \end{cases}$
  \item $\lvert E_{i-1}\rvert$ even, $a_{i-1}>0$ $\implies$
    $\delta_i=\delta_{i-1}-
    \begin{cases}
\max\{0,\,x_{i-1}-x_i,\,(x_i-1)\} & \text{if }x_{i-1}\leqslant x_1,\\
\max\{0,\,x_1-x_i,\,x_i-x_{i-1}\}  & \text{if }x_{i-1}\geqslant x_1.    
    \end{cases}$
  \item $\lvert E_{i-1}\rvert$ even, $a_{i-1}<0$ $\implies$ $\delta_i=\delta_{i-1}-
    \begin{cases}
\max\{0,\,1-x_{1}-x_i,\,x_{i-1}+x_i-1\} & \text{if }x_{i-1}\leqslant x_1,\\
\max\{0,\,1-x_{i-1}-x_i,\,x_1+x_i-1\}  & \text{if }x_{i-1}\geqslant x_1.    
    \end{cases}$\qedhere
  \end{enumerate}
\end{proof}
\begin{proof}[Proof of Lemma~\ref{lem:cycle_UB}] 
  If $\lvert E^-\rvert$ is even then $\delta_2=0$, hence $\delta_i=0$ for all $i\in[2,n]$ by
  Lemma~\ref{lem:mono}, and there is nothing to do. For odd $\lvert E^-\rvert$ we proceed by
  induction on $n$. Note that the construction of the sets $X_i$, the definition of the numbers
  $\delta_i$, and the statement of the lemma depend only on sequences $(x_1,\dotsc,x_n)$ and
  $(a_1,\dotsc,a_n)$ (actually only on the signs of the $a_i$). Therefore we can use $n=2$ as the
  base case. Then $\lvert E^-\rvert=1$, $V^+=V^-=\emptyset$, $A=0$, and we have the following two cases.
\begin{description}
\item[Case 1] if $a_2<0<a_1$, then $\mu( X_1\cap X_2)=\min\{x_1,x_2\}$ and
\[\delta_2=\mu( X_1\cap X_2)-\max\{0,x_1+x_2-1\}=\mu_1-\eta_2.\]
\item[Case 2] if $a_1<0<a_2$, then $\mu( X_1\cap X_2)=\max\{0,x_1+x_2-1\}$ and
\[\delta_2=\min\{x_1,x_2\}-\mu( X_1\cap X_2)=\mu_2-\eta_1.\]
\end{description}
Now let $n\geqslant 3$ and set
\[\gamma=\sum_{i\in E^+}\mu_i-\sum_{i\in E^-}\eta_i-A=\sum_{i\in E^+}\mu_i-\sum_{i\in
    E^-}\eta_i-x(V^+)+x(V^-)-\left\lfloor\frac{\lvert E^-\rvert}{2}\right\rfloor,\] so that our aim
becomes to show $\delta_n=\gamma$.  Applying the induction hypothesis to the sets
$X_1,\dots,X_{n-1}$ that are obtained by applying the construction for the sequences
$(x_1,\dots,x_{n-1})$ and $(a_1,\dots,a_{n-2},\sign(a_{n-1}a_n))$ (the number of negative terms is
still odd), we get
\[\delta_{n-1}=
\begin{cases}
  \gamma-\mu_{n-1}-\mu_{n}+x_n+\min\{x_{n-1},\,x_1\} & \text{if }n-1,n\in E^+,\\ 
    \gamma+\eta_{n-1}+\eta_{n}-x_n-x_{n-1}-x_1+\min\{x_{n-1},\,x_1\}+1 & \text{if }n-1,n\in E^-,\ \\ 
  \gamma-\mu_{n-1}+\eta_{n}+x_{n-1}-\max\{0,\,x_{n-1}+x_1-1\} & \text{if }n-1\in E^+,\,n\in E^-,\\ 
   \gamma+\eta_{n-1}-\mu_{n}+x_1-\max\{0,\,x_{n-1}+x_1-1\} & \text{if }n-1\in E^-,\,n\in E^+. 
\end{cases}
\]
Now we discuss the four cases separately. The detailed case analysis can be found in Appendix~\ref{sec:det-cycle_UB}.

\begin{description}
\item[Case 1] $n-1,n\in E^+$. We may assume $\mu( X_1\cap X_{n-1})<x_n<x_1+x_{n-1}-\mu( X_1\cap X_{n-1})$ since otherwise $\delta_n=0$.
  \begin{description}
  \item[Case 1.1] $x_{n-1}\leqslant x_1$. In this case $\delta_{n-1}=\mu( X_{n-1}\setminus
    X_{1})$. 
  \item[Case 1.2] $x_{n-1}>x_1$. In this case $\delta_{n-1}=\mu( X_1\setminus X_{n-1})$. 
  \end{description}
\item[Case 2] $n-1,n\in E^-$. We may assume $1-\mu( X_1\cup X_{n-1})<x_n<1-\mu( X_1\cap X_{n-1})$ since otherwise $\delta_n=0$.
  \begin{description}
  \item[Case 2.1] $x_{n-1}\leqslant x_1$. In this case $\delta_{n-1}=\mu( X_{n-1}\setminus
    X_{1})=\mu( X_1\cup X_{n-1})-x_1$. 
  \item[Case 2.2] $x_{n-1}>x_1$. In this case $\delta_{n-1}=\mu( X_1\setminus X_{n-1})=\mu( X_1\cup X_{n-1})-x_{n-1}$. 
  \end{description}
\item[Case 3] $n-1\in E^+,\,n\in E^-$. We may assume $\mu( X_{n-1}\setminus X_1)<x_n<1-\mu( X_{1}\setminus X_{n-1})$ since otherwise $\delta_n=0$.
  \begin{description}
  \item[Case 3.1] $x_1+x_{n-1}\leqslant 1$. In this case $\delta_{n-1}=\mu( X_1\cap
    X_{n-1})$.
  \item[Case 3.2] $x_1+x_{n-1}> 1$. In this case $\delta_{n-1}=\mu([0,1)\setminus(X_1\cup
    X_{n-1}))$.
  \end{description}
\item[Case 4] $n-1\in E^-,\,n\in E^+$. We may assume $\mu( X_{n-1}\setminus X_1)<x_n<1-\mu( X_{n-1}\setminus X_1)$ since otherwise $\delta_n=0$.
  \begin{description}
  \item[Case 4.1] $x_1+x_{n-1}\leqslant 1$. In this case $\delta_{n-1}=\mu( X_1\cap
    X_{n-1})$.
  \item[Case 4.2] $x_1+x_{n-1}> 1$. In this case $\delta_{n-1}=\mu([0,1)\setminus(X_1\cup
    X_{n-1}))$.\qedhere
  \end{description}
\end{description}
\end{proof}
Finally, to complete our proof of Theorem~\ref{thm:cycles}, we need to verify necessity of the McCormick inequalities~\eqref{eq:mccormick} and the inequalities~\eqref{eq:cycle_1} and~\eqref{eq:cycle_2}. For each of these inequalities, we exhibit a point $(\vect x,\vect y)\in P'$, where $P^{\prime}$ is the polytope obtained by dropping an inequality from $P$, such that $\pi[f](\vect x,\vect y)\not\in X(f)$.


\subsubsection*{The inequality $y_i\leq x_i$}
  Let $x_{i+1}=1$ and $x_j=0$ for all $j\in[n]\setminus\{i+1\}$, so that
  $(\vect x,z)\in X(f)\iff z=0$. Setting $y_i=1$ and $y_j=0$ for all $j\in[n]\setminus\{i\}$, we
  obtain a point $(\vect x,\vect y)$ with $\pi[f](\vect x,\vect y)=(\vect x,a_i)\not\in X(f)$. If
  $E^-\neq\{i+1\}$ and $E^+\neq\{i+1\}$, then $(\vect x,\vect y)$ satisfies~\eqref{eq:cycle_1}
  and~\eqref{eq:cycle_2}, so $(\vect x,\vect y)\in P'$, as required. But if $E^-=\{i+1\}$
  (resp. $E^+=\{i+1\}$) then $(\vect x,\vect y)$ is cut off by~\eqref{eq:cycle_1} (resp.~\eqref{eq:cycle_2}). In these cases we
  use the point given by $x_{i+1}=x_{i+2}=1$, $x_j=0$ for all $j\in[n]\setminus\{i+1,i+2\}$,
  $y_i=y_{i+1}=1$ and $y_j=0$ for all $j\in[n]\setminus\{i,i+1\}$ instead.

  The inequality $y_i\leq x_{i+1}$ can be treated similarly.
  
\subsubsection*{The inequality $y_i\geq x_i+x_{i+1}-1$}
Let $x_{i}=x_{i+1}=1$ and $x_j=0$ for all $j\in[n]\setminus\{i,i+1\}$, so that $(\vect
    x,z)\in X(f)\iff z=a_i$. Setting $\vect y=\vect 0$, we obtain
    a point $(\vect x,\vect y)$ with $\pi[f](\vect x,\vect y)=(\vect x,0)\not\in
    X(f)$. If $E^+\neq\{i-1,i,i+1\}$ and $E^-\neq\{i-1,i,i+1\}$ then $(\vect x,\vect y)\in P'$ and we are
    done. Otherwise, we use the point given by $x_{i}=x_{i+1}=x_{i+2}=1$, $x_j=0$ for all
    $j\in[n]\setminus\{i,i+1,i+2\}$, $y_{i+1}=1$ and $y_j=0$ for all $j\in[n]\setminus\{i+1\}$.  
    
\subsubsection*{The inequality~\eqref{eq:cycle_1}}
This is Padberg's cycle inequality corresponding to the subset $D = E^{-}$, and hence is known to be implied by the McCormick inequalities when  $\lvert E^{-}\rvert$ is even. Now let $\lvert E^-\rvert$ be odd, say $\lvert E^-\rvert=2k+1$ for some non-negative integer $k$. Set $x_i=1/2$ for all $i\in[n]$. Then $(\vect x,\vect y)\in P$ implies
  \begin{multline*}
    y(E^+)-y(E^-)\stackrel{\eqref{eq:cycle_1}}{\leq} x(V^+)-x(V^-)+k=\frac{\lvert V^+\rvert-\lvert
      V^-\rvert}{2}+k=\frac{n-2(2k+1)}{2}+k\\
    =\frac{n-2k-2}{2} = \frac{\lvert E^+\rvert-1}{2},
  \end{multline*}
and consequently, for every $(\vect x,z)\in X(f)$, \[z=\sum_{i=1}^na_iy_i=\sum_{i\in E^+}a_iy_i+\sum_{i\in E^-}a_iy_i<\frac12\sum_{i\in E^+}a_i.\]  
Setting $y_i=1/2$ for all $i\in E^+$ and $y_i=0$ for all $i\in E^-$, we obtain a point $(\vect x,\vect y)\in P'$ with $\pi[f](\vect x,\vect y)=\left(\vect x,\,\frac12\sum_{i\in E^+}a_i\right)\not\in  X(f)$.

\subsubsection*{The inequality~\eqref{eq:cycle_2}}
The arguments here are similar to those for \eqref{eq:cycle_1}.

\subsection{Proof of Theorem~\ref{thm:cactus}}\label{subsec:cactus}
Let $G=(V,E)$ be a cactus graph with $k$ cycles and arbitrary edge weights. We want to show that for
the corresponding function $f$, $\X(f)=\pi[f](P)$ where $P$ is described by the McCormick
inequalities and at most $2k$ cycle inequalities. We prove this by induction on $k$, the number of
cycles. For $k=0$, $G$ is a tree, and the McCormick inequalities are sufficient. For $k=1$, we
proceed by induction on the number of edges that are not contained in the cycle. If there are no
such edges then $G$ is a cycle, and the claim follows from Theorem~\ref{thm:cycles}. Otherwise,
there are two graphs $G_1=(V_1,E_1)$ and $G_2=(V_2,E_2)$ with $V_1\cup V_2=V$, $E_1\cup E_2=E$,
$\lvert V_1\cap V_2\rvert=1$, $E_2\neq\emptyset$, and such that the cycle of $G$ is contained in
$G_1$. Let $f_1$ and $f_2$ be the bilinear functions corresponding to the graphs $G_1$ and $G_2$,
respectively. Since $G_2$ is cycle-free, $X(f_2)$ is described by the McCormick inequalities, and by
induction, $X(f_1)$ is described by the McCormick inequalities and at most two cycle inequalities. Now
the result follows from Corollary~\ref{cor:combination}. For $k\geq 2$, there are two graphs
$G_1=(V_1,E_1)$ and $G_2=(V_2,E_2)$ with $V_1\cup V_2=V$, $E_1\cup E_2=E$,
$\lvert V_1\cap V_2\rvert=1$, such that $k=k_1+k_2$ where $k_i$ ($i\in\{1,2\}$) is the number of
cycles in graph $G_i$. Again, let $f_1$ and $f_2$ be the bilinear functions corresponding to the graphs
$G_1$ and $G_2$, respectively. By induction, $X(f_1)$ is described by the McCormick inequalities and
at most $2k_1$ cycle inequalities, and $X(f_2)$ is described by the McCormick inequalities and
at most $2k_2$ cycle inequalities. The result follows from Corollary~\ref{cor:combination}.

\section{Conclusion  and open problems}

We have used an extension of Zuckerberg's geometric method for characterizing convex hulls of
subsets of the discrete $n$-cube to find extended formulations for the convex hulls of graphs of
bilinear functions corresponding to almost complete graphs with unit weights and cactus graphs with
arbitrary weights. We think that this approach can be used in more general situations, but
additional insights are needed to avoid the tedious case discussions as in the proof of Theorem~\ref{thm:cycles}.

A natural next test case for the method is the class of wheels. A wheel $W_{n-1}$ is the graph
with vertex set $V=\{1,\dots,n\}$ for $n\geq 5$, and edge set
\[E = \{\{1,2\}, \{2,3\}, \dots, \{n-2,n-1\} \} \cup\{\{i,n\}\,:\,1\leq i\leq n-1\}.\]
That is, $W_{n-1}$ is the cycle $C_{n-1}$ with spokes from the center vertex $n$. Since $W_{n-1}$ has a
$K_{4}$-minor, we know that $\QP(W_{n})$ needs more than the cycle inequalities for its
description. One extra facet-defining inequality for $\QP(W_{n-1})$ is
\begin{equation}\label{eq:wheelA}
\left\lfloor (n-1)/2 \right\rfloor x_{n} \,+\, \sum_{i=1}^{n}x_{i} \,-\, y(E) \le \left\lfloor (n-1)/2 \right\rfloor.
\end{equation}
which can be argued by first principles. Another inequality comes from the cut polytope of a
graph. \citet{barahona1986cut} introduced $\cut(G)$ as the convex hull of incidence vectors of the
cuts in $G$, and showed that for every odd bicycle wheel in $G$, there is a corresponding
facet-defining inequality for $\cut(G)$. An \emph{odd bicycle wheel} is the graph $W_{n-1} + \{v\}$
for odd $n-1$, where $G+\{v\}$ is the graph obtained from $G$ by joining every vertex of $G$ to a
new vertex $v$. When $G=W_{n-1}$ and $n-1$ is odd, then $G+\{n+1\}$ is an odd bicycle wheel and we
have exactly one odd bicycle wheel inequality for $\cut(W_{n-1}+\{n+1\})$. Mapping this inequality
to $\QP(W_{n-1})$ using the well-known linear bijection between $\QP(G)$ and $\cut(G + \{v\})$
\citep{de1990cut} leads to the inequality
\begin{equation}\label{eq:wheelB}
\left(\frac{n}{2}\right)x_{n} \,+\, 2\sum_{i=1}^{n-1}x_{i} \,-\, y(E) \le n-1
\end{equation}
defining a facet of $\QP(W_{n-1})$ when $n$ is odd. When $n-1=5$, the inequalities~\eqref{eq:wheelA}
and~\eqref{eq:wheelB} are sufficient to convexify $f$.

\begin{observation}[5-wheel]
  If $G = W_5$ and all edge weights are equal to $1$, then $\X(f) = \pi[f](\P)$ where $P$ is the
  polytope described by the McCormick inequalities~\eqref{eq:mccormick} together with
  \begin{align}
    2 x_6 + x_1 + \dotsb + x_5 - y(E) &\leq 2,\label{eq:wheel_1}\\
    3 x_6 + 2(x_1 + \dotsb + x_5) - y(E) &\leq 5,\label{eq:wheel_2}
  \end{align}
which are are exactly inequalities \eqref{eq:wheelA} and \eqref{eq:wheelB} for $n = 6$.
\end{observation}
We conjecture that this observation extends to any wheel $W_{n-1}$ with $n\ge 6$, $n$ even.

The open question on wheel graphs can be extended to a richer family of graphs called \emph{Halin
  graphs} which was introduced by \citet{halin1971studies}. A Halin graph is a planar graph obtained
from a tree without vertices of degree 2 by adding a cycle through all the leaves. Thus, a wheel is
the simplest kind of a Halin graph. The insights gained from generalizing the above observation on
wheel graphs might be useful to characterizing a polytope $\P$ that does not project onto $\QP(G)$
yet has $\X(f) = \pi[f](\P)$ and is a minimal such extension of $\X(f)$, when $G$ is a Halin graph.

\subsection*{Acknowledgements}
The first author thanks Jon Lee for bringing cactus graphs to his attention during stimulating
discussions at the Dagstuhl Seminar on Designing and Implementing Algorithms for Mixed-Integer
Nonlinear Optimization held at Schloss Dagstuhl in February 2018.

{
\newrefcontext[sorting=nyt]
\printbibliography
}

\newpage
\appendix

\section{Minimality proof for the polytope in Theorem~\ref{thm:clique_minus} }\label{sec:minimality_proof}
We want to show that each of the inequalities listed in Theorem~\ref{thm:clique_minus} is necessary in the
sense that omitting it leads to a polytope $P'$ with $\pi[f](P')\supsetneq X(f)$.
\subsection{The McCormick inequalities $y_{ij}\leq x_i$}
Let $P'$ be the polytope obtained from $P$ by omitting the
inequality $y_{i^*j^*}\leq x_{i^*}$, and consider the point $\vect x$ with $x_{i^*}=0$ and $x_i=1$
for all $i\in V\setminus\{i^*\}$. Then $(\vect x,z)\in X(f)$ if and only if
\[z=\left\lvert E(V\setminus\{i^*\})\right\rvert=
  \begin{cases}
    \binom{n-1}{2}-1 &\text{if }i^*\in V\setminus\{n-1,n\},\\
    \binom{n-1}{2} &\text{if }i^*\in \{n-1,n\}.
  \end{cases}
\]
If $\{i^*,j^*\}\neq\{n-1,n\}$ then we obtain a point $(\vect x,\vect y)\in P'$ by setting $y_{ij}=1$
for all $ij\in E(V\setminus\{i^*\})$, $y_{i^*j^*}=1$ and $y_{i^*j}=0$ for $j\neq j^*$. Then
$\pi[f](\vect x,\vect y)\not\in X(f)$ because $y(E)=\binom{n-1}{2}+1$. For
$i^*=n-1$, $j^*=n$, we obtain $(\vect x,\vect y)\in P'$ by setting $y_{in}=1$ and $y_{i,n-1}=0$ for
all $i\in V\setminus\{n-1,n\}$, $y_{n-1,n}=1$, and $y_{ij}=1-1/\binom{n-2}{2}$ for all $ij\in E(V\setminus\{n-1,n\})$.
Then $\pi[f](\vect x,\vect y)\not\in X(f)$ because
\[\sum_{ij\in E}y_{ij}=\binom{n-2}{2}-1+(n-2)=\binom{n-1}{2}-1.\]
\subsection{The inequalities~\eqref{eq:clique_minus_0}}
Without loss of generality, $i=1$. We obtain a point
$(\vect x,\vect y)\in P'$ by setting
\begin{align*}
  x_{n-1}=x_n=y_{n-1,n}&=1/2, & x_1=x_2=y_{12}&=5/6, & x_3=x_4=\dots=x_{n-2} &=0,\\
  y_{1,n-1}=y_{1n} &=0, & y_{2,n-1}=y_{2n} &=1/2,
\end{align*}
and $y_{ij}=0$ for all remaining $ij\in E$. Then $\pi[f](\vect x,\vect y)=(\vect x,\,11/6)\not\in X(f)$, because $(\vect x,z)\in
X(f)$ implies
\begin{multline*}
  z\geq(x_1+x_2-1)+(x_1+x_{n-1}-1)+(x_1+x_{n}-1)+(x_2+x_{n-1}-1)+(x_2+x_{n}-1)\\
  =2/3+4(1/3)=2.
\end{multline*}
\subsection{The inequalities~\eqref{eq:clique_minus_1}}
Let $t\in[n-2]$, and consider the point $(\vect x,\vect y)$ given by
\begin{align*}
  x_{n-1}=x_n=y_{n-1,n}&=1/2, \\
  x_{i} &=
          \begin{cases}
            1 &\text{for }1\leq i\leq t,\\
            0&\text{for }t+1\leq i\leq n-2,\\
          \end{cases}\\
  y_{i,n-1}=y_{in} &=
                     \begin{cases}
                       1/2 &\text{for }1\leq i\leq t,\\
                       0&\text{for }t+1\leq i\leq n-2,\\
                     \end{cases}\\
  y_{ij} &=1-\frac{2}{t(t-1)} &&\text{for }1\leq i<j\leq t,
\end{align*}
and $y_{ij}=0$ for all remaining $ij\in E$. Then $(\vect x,\vect y)\in P'$, where $P'$ is the
polytope obtained from $P$ by omitting~\eqref{eq:clique_minus_1} for $s=t$, and $\pi[f](\vect x,\vect y)=(\vect x,\,\binom{s+1}{2}-1)\not\in X(f)$, because $(\vect x,z)\in
X(f)$ implies
\[z\geq\binom{s}{2}+s=\binom{s+1}{2}.\]
\subsection{The inequalities~\eqref{eq:clique_minus_3}}
The following lemma shows that~\eqref{eq:clique_minus_3} cannot be omitted for any $s<(n-1)/2$.
\begin{lemma}\label{lem:violation_1}
  Let $t$ be an integer with $1\leq t<(n-1)/2$, and consider the point
  $(\vect x,\vect y)\in\reals^{n(n+1)/2}$ given by $x_{n-1}=1$, $x_n=y_{n-1,n}=0$,
\begin{align*}
  x_i &=\frac{t-1/2}{n-2} && \text{for }1\leq i\leq n-2,\\
  y_{i,n-1}=y_{in}&=0 && \text{for }i\leq n-2,\\
  y_{ij} &=\frac{t^2-1}{(n-2)(n-3)} && \text{for }1\leq i<j\leq n-2.
\end{align*}
Then $(\vect x,\vect y)$ satisfies all the constraints describing $P$
except~(\ref{eq:clique_minus_3}) for $s=t$. Moreover,
\[\pi[f](\vect x,\vect y)=\left(\vect x,\,(t^2-1)/2\right)\not\in X(f).\]
\end{lemma}
\begin{proof}
  All $y$-variables are non-negative so $y(E)\geq 0$. For the inequalities
  $y_{ij}\leq\min\{x_i,x_j\}$, we just need to check that
  \[\frac{t^2-1}{(n-2)(n-3)}\leq\frac{t-1/2}{n-2},\]
  or equivalently, $\phi(t)\leq 0$ where $\phi$ is the quadratic function
  $\phi(t)=(t^2-1)-(t-1/2)(n-3)$. This follows from $\phi(1)=(3-n)/2<0$ and
  $\phi((n-1)/2)=-(n-3)(n-5)/4\leq 0$.

  For~(\ref{eq:clique_minus_0}), we use the assumption $2t<n-1$:
  \[2x_i+x_{n-1}+x_n-y_{i,n-1}-y_{in}=\frac{2t-1}{n-2}+1< 2.\]

  For~(\ref{eq:clique_minus_1}), we use
  \begin{align*}
    x(V\setminus\{n-1,n\})+\frac{x_{n-1}+x_n}{2} &= t,\\
    y(E(V\setminus\{n-1,n\}))+\frac12\sum_{i=1}^{n-2}\left(y_{i,n-1}+y_{in}\right) &= \frac{t^2-1}{2}.
  \end{align*}
  Now~(\ref{eq:clique_minus_1}) follows from
  \[st-\frac{t^2-1}{2}-\binom{s+1}{2}=-\frac12\left[\left(s-\frac{2t-1}{2}\right)^2+t-\frac54\right]\leq 0\]
  for all integers $s$.

  For (\ref{eq:clique_minus_3}), we we have
  \begin{multline*}
    sx(V)-y(E)-y_{n-1,n}-\binom{s+1}{2} =
    s\left(t+\frac12\right)-\frac{t^2-1}{2}-\binom{s+1}{2}\\
    =-\frac12(s-t+1)(s-t-1)\ 
    \begin{cases}
      \leq 0 &\text{for }s\neq t,\\
      =1/2 &\text{for }s=t.
    \end{cases}
  \end{multline*}
  Finally, $\left(\vect x,\,(t^2-1)/2\right)\not\in X(f)$ because
  \[tx(V)-z-x_n-\binom{t+1}{2}=t\left(t+\frac12\right)-\frac{t^2-1}{2}-\binom{t+1}{2}=\frac12\]
  while
  \[tx(V)-z-x_n-\binom{t+1}{2}\leq 0\]
  is a valid inequality for $X(f)$.
\end{proof}
To conclude the proof, the next lemma shows that~\eqref{eq:clique_minus_3} cannot be omitted for any $s\geq
(n-1)/2$.
\begin{lemma}\label{lem:violation_2}
    Let $t$ be an integer with $(n-1)/2\leq t\leq n-2$, and consider the point $(\vect x,\vect
    y)\in\reals^{n(n+1)/2}$ given by $x_{n-1}=1$, $x_n=y_{n-1,n}=1/2$,
\begin{align*}
   x_i &=\frac{t-1}{n-2} && \text{for }1\leq i\leq n-2,\\
   y_{i,n-1}&=\frac{4t-n-2}{3(n-2)} \text{ and } y_{i,n}=\frac{4t-n-2}{6(n-2)} &&\text{for }i\leq n-2,\\
  y_{ij} &=\frac{2t^2-8t+2n+1}{2(n-2)(n-3)}   &&\text{for }1\leq i<j\leq n-2.
\end{align*}
Then $(\vect x,\vect y)$ satisfies all the constraints describing $P$
except~(\ref{eq:clique_minus_3}) for $s=t$. Moreover,
\[\pi[f](\vect x,\vect y)=\left(\vect x,\,(2t^2-3)/4\right)\not\in X(f).\]
\end{lemma}
\begin{proof}
  All $y$-variables are non-negative so $y(E)\geq 0$. For the inequalities
  $y_{ij}\leq\min\{x_i,x_j\}$, we need to check that
  \begin{align*}
    \frac{2t^2-8t+2n+1}{2(n-2)(n-3)} &\leq \frac{t-1}{n-2}, & \frac{4t-n-2}{3(n-2)} &\leq \frac{t-1}{n-2}.
  \end{align*}
  The first inequality is equivalent to $\phi(t)\leq 0$ where $\phi(t)=2t^2-8t+2n+1-2(n-3)(t-1)$,
  and this follows from
  \begin{align*}
    \phi\left(\frac{n-1}{2}\right)&=-\frac12(n^2-6n+7)\leq 0, & \phi\left(n-2\right)&=7-2n\leq 0.
  \end{align*}
  The second inequality follows from $4t-n-2-3(t-1)=t+1-n<0$.

  For~\eqref{eq:clique_minus_0}, we have
  \[2x_i+x_{n-1}+x_n-y_{i,n-1}-y_{in}=\frac{2t-2}{n-2}+\frac32-\frac{4t-n-2}{2(n-2)}=2.\]

  For~\eqref{eq:clique_minus_1}, we use
  \begin{align*}
    x(V\setminus\{n-1,n\})+\frac{x_{n-1}+x_n}{2} &= t-\frac14,\\
    y(E(V\setminus\{n-1,n\}))+\frac12\sum_{i=1}^{n-2}\left(y_{i,n-1}+y_{in}\right) &=
                                                                                     \frac{2t^2-8t+2n+1}{4}+\frac{4t-n-2}{4}\\
    &=\frac{2t^2-4t+n-1}{4}.
  \end{align*}
  Now~(\ref{eq:clique_minus_1}) follows from
  \begin{multline*}
    s\left(t-\frac14\right)-\frac{2t^2-4t+n-1}{4}-\binom{s+1}{2}=-\frac12\left[\left(s-\frac{4t-3}{4}\right)^2+\frac{n-t}{2}-\frac{17}{16}\right]\\
    \leq -\frac12\left[\left(s-\frac{4t-3}{4}\right)^2-\frac{1}{16}\right]\leq 0.
  \end{multline*}
  for all integers $s$.

  For (\ref{eq:clique_minus_3}), we have
  \[y(E)+y_{n,n-1}=\frac{2t^2-8t+2n+1}{4}+\frac{4t-n-2}{2}+\frac12=\frac{2t^2-1}{4},\]
  and then
  \begin{multline*}
    sx(V)-y(E)-y_{n-1,n}-\binom{s+1}{2} =
    s\left(t+\frac12\right)-\frac{2t^2-1}{4}-\binom{s+1}{2}\\
    =-\frac12\left[(s-t)^2-\frac12\right]
    \begin{cases}
      \leq 0 &\text{for }s\neq t,\\
      =1/4 &\text{for }s=t.
    \end{cases}
  \end{multline*}
  Finally, $\left(\vect x,\,(2t^2-3)/4\right)\not\in X(f)$ because
  \[tx(V)-z-x_n-\binom{t+1}{2}=t\left(t+\frac12\right)-\frac{2t^2-3}{4}-\binom{t+1}{2}=\frac34\]
  while
  \[tx(V)-z-x_n-\binom{t+1}{2}\leq 0\]
  is a valid inequality for $X(f)$.
\end{proof}

\section{Detailed case analysis in the proof of Lemma~\ref{lem:cycle_UB}}\label{sec:det-cycle_UB}

\begin{description}
\item[Case 1] $n-1,n\in E^+$. We may assume $\mu( X_1\cap X_{n-1})<x_n<x_1+x_{n-1}-\mu( X_1\cap X_{n-1})$ since otherwise $\delta_n=0$.
  \begin{description}
  \item[Case 1.1] $x_{n-1}\leqslant x_1$. In this case $\delta_{n-1}=\mu( X_{n-1}\setminus
    X_{1})$. 

If $x_n\leqslant x_{n-1}$ then 
    \begin{align*}
  \delta_n&=\mu( X_n\setminus X_{1})=\delta_{n-1}-(x_{n-1}-x_n)\\
     &=\gamma-\mu_{n-1}-\mu_{n}+x_n+\min\{x_{n-1},\,x_1\}-(x_{n-1}-x_n)\\
     &=\gamma-x_n-x_{n}+x_n+x_{n-1}-(x_{n-1}-x_n)=\gamma.
    \end{align*}
If $x_{n-1}\leqslant x_n\leqslant x_1$ then 
    \begin{align*}
  \delta_n&=\mu( X_{n}\setminus X_1)=\delta_{n-1}\\
     &=\gamma-\mu_{n-1}-\mu_{n}+x_n+\min\{x_{n-1},\,x_1\}\\
     &=\gamma-x_{n-1}-x_{n}+x_n+x_{n-1}=\gamma.
    \end{align*}
Finally, if $x_n\geqslant x_1$ then 
\begin{align*}
  \delta_n&=\mu( X_1\setminus X_n)=\delta_{n-1}-(x_n-x_1)\\
     &=\gamma-\mu_{n-1}-\mu_{n}+x_n+\min\{x_{n-1},\,x_1\}-(x_n-x_1)\\
     &=\gamma-x_{n-1}-x_{1}+x_n+x_{n-1}-(x_n-x_1)=\gamma.
    \end{align*}
  \item[Case 1.2] $x_{n-1}>x_1$. In this case $\delta_{n-1}=\mu( X_1\setminus X_{n-1})$. 

If $x_n\leqslant x_{1}$ then 
    \begin{align*}
  \delta_n&=\mu( X_n\setminus X_{1})=\delta_{n-1}-(x_{1}-x_n)\\
     &=\gamma-\mu_{n-1}-\mu_{n}+x_n+\min\{x_{n-1},\,x_1\}-(x_{1}-x_n)\\
     &=\gamma-x_n-x_{n}+x_n+x_{1}-(x_{1}-x_n)=\gamma.
    \end{align*}
If $x_{1}\leqslant x_n\leqslant x_{n-1}$ then 
    \begin{align*}
  \delta_n&=\mu( X_{1}\setminus X_n)=\delta_{n-1}\\
     &=\gamma-\mu_{n-1}-\mu_{n}+x_n+\min\{x_{n-1},\,x_1\}\\
     &=\gamma-x_{n}-x_{1}+x_n+x_{1}=\gamma.
    \end{align*}
Finally, if $x_n\geqslant x_{n-1}$ then 
\begin{align*}
  \delta_n&=\mu( X_1\setminus X_n)=\delta_{n-1}-(x_n-x_{n-1})\\
     &=\gamma-\mu_{n-1}-\mu_{n}+x_n+\min\{x_{n-1},\,x_1\}-(x_n-x_{n-1})\\
     &=\gamma-x_{n-1}-x_{1}+x_n+x_{1}-(x_n-x_{n-1})=\gamma.
    \end{align*}   
  \end{description}
\item[Case 2] $n-1,n\in E^-$. We may assume $1-\mu( X_1\cup X_{n-1})<x_n<1-\mu( X_1\cap X_{n-1})$ since otherwise $\delta_n=0$.
  \begin{description}
  \item[Case 2.1] $x_{n-1}\leqslant x_1$. In this case $\delta_{n-1}=\mu( X_{n-1}\setminus
    X_{1})=\mu( X_1\cup X_{n-1})-x_1$. 

If $x_n\leqslant 1-x_1$ then 
    \begin{align*}
  \delta_n&=\mu( X_n\cap X_{1})=x_n-(1-\mu( X_1\cup X_{n-1}))=\delta_{n-1}-(1-x_1-x_n)\\
     &=\gamma+\eta_{n-1}+\eta_{n}-x_n-x_{n-1}-x_1+\min\{x_{n-1},\,x_1\}+1-(1-x_1-x_n)\\
     &=\gamma+0+0-x_n-x_{n-1}-x_1+x_{n-1}+1-(1-x_1-x_n)=\gamma.
    \end{align*}
If $1-x_1\leqslant x_n\leqslant 1-x_{n-1}$ then 
    \begin{align*}
  \delta_n&=\mu([0,1]\setminus (X_{n}\cup X_1))=\mu( X_{n-1}\setminus
    X_{1})=\delta_{n-1}\\
     &=\gamma+\eta_{n-1}+\eta_{n}-x_n-x_{n-1}-x_1+\min\{x_{n-1},\,x_1\}+1\\
     &=\gamma+0+(x_n+x_1-1)-x_n-x_{n-1}-x_1+x_{n-1}+1=\gamma.
    \end{align*}
Finally, if $x_n\geqslant 1-x_{n-1}$ then 
\begin{align*}
  \delta_n&=\mu([0,1]\setminus(X_{n}\cup X_1))=\delta_{n-1}-(x_n+x_{n-1}-1)\\
     &=\gamma+\eta_{n-1}+\eta_{n}-x_n-x_{n-1}-x_1+\min\{x_{n-1},\,x_1\}+1-\eta_{n-1}\\
     &=\gamma+\eta_{n-1}+(x_n+x_1-1)-x_n-x_{n-1}-x_1+x_{n-1}+1-\eta_{n-1}=\gamma.
    \end{align*}
  \item[Case 2.2] $x_{n-1}>x_1$. In this case $\delta_{n-1}=\mu( X_1\setminus X_{n-1})=\mu( X_1\cup X_{n-1})-x_{n-1}$. 

If $x_n\leqslant 1-x_{n-1}$ then 
    \begin{align*}
  \delta_n&=\mu( X_n\cap X_{1})=x_n-(1-\mu( X_1\cup X_{n-1}))=\delta_{n-1}-(1-x_n-x_{n-1})\\
     &=\gamma+\eta_{n-1}+\eta_{n}-x_n-x_{n-1}-x_1+\min\{x_{n-1},\,x_1\}+1-(1-x_n-x_{n-1})\\
     &=\gamma+0+0-x_n-x_{n-1}-x_1+x_{1}+1-(1-x_n-x_{n-1})=\gamma.
    \end{align*}
If $1-x_{n-1}\leqslant x_n\leqslant 1-x_{1}$ then 
    \begin{align*}
  \delta_n&=\mu( X_{n}\cap X_1)=\mu( X_{1}\setminus
    X_{n-1})=\delta_{n-1}\\
     &=\gamma+\eta_{n-1}+\eta_{n}-x_n-x_{n-1}-x_1+\min\{x_{n-1},\,x_1\}+1\\
     &=\gamma+(x_{n-1}+x_n-1)+0-x_n-x_{n-1}-x_1+x_{1}+1=\gamma.
    \end{align*}
Finally, if $x_n\geqslant 1-x_{1}$ then 
\begin{align*}
  \delta_n&=\mu([0,1]\setminus (X_{n}\cup X_1))=\delta_{n-1}-(x_n+x_{1}-1)\\
     &=\gamma+\eta_{n-1}+\eta_{n}-x_n-x_{n-1}-x_1+\min\{x_{n-1},\,x_1\}+1-\eta_{n}\\
     &=\gamma+(x_{n-1}+x_n-1)+\eta_n-x_n-x_{n-1}-x_1+x_{1}+1-\eta_{n}=\gamma.
    \end{align*}
  \end{description}
\item[Case 3] $n-1\in E^+,\,n\in E^-$. We may assume $\mu( X_{n-1}\setminus X_1)<x_n<1-\mu( X_{1}\setminus X_{n-1})$ since otherwise $\delta_n=0$.
  \begin{description}
  \item[Case 3.1] $x_1+x_{n-1}\leqslant 1$. In this case $\delta_{n-1}=\mu( X_1\cap
    X_{n-1})$.

If $x_n\leqslant x_{n-1}$ then 
    \begin{align*}
  \delta_n&=\delta_{n-1}-(x_{n-1}-x_n)\\
     &=\gamma-\mu_{n-1}+\eta_{n}+x_{n-1}-\max\{0,\,x_{n-1}+x_1-1\}-(x_{n-1}-x_n)\\
     &=\gamma-x_n+0+x_{n-1}-0-(x_{n-1}-x_n)=\gamma.
    \end{align*}
If $x_{n-1}\leqslant x_n\leqslant 1-x_{1}$ then 
    \begin{align*}
  \delta_n&=\delta_{n-1}\\
     &=\gamma-\mu_{n-1}+\eta_{n}+x_{n-1}-\max\{0,\,x_{n-1}+x_1-1\}\\
     &=\gamma-x_{n-1}+0+x_{n-1}-0=\gamma.
    \end{align*}
Finally, if $x_n\geqslant 1-x_{1}$ then 
\begin{align*}
  \delta_n&=\delta_{n-1}-(x_1+x_{n}-1)\\
     &=\gamma-\mu_{n-1}+\eta_{n}+x_{n-1}-\max\{0,\,x_{n-1}+x_1-1\}-(x_1+x_{n}-1)\\
     &=\gamma-x_{n-1}+(x_1+x_n-1)+x_{n-1}-0-(x_1+x_{n}-1)=\gamma.
    \end{align*}
  \item[Case 3.2] $x_1+x_{n-1}> 1$. In this case $\delta_{n-1}=\mu([0,1]\setminus(X_1\cup
    X_{n-1}))$.

If $x_n\leqslant 1-x_{1}$ then 
    \begin{align*}
  \delta_n&=\delta_{n-1}-(1-x_1-x_n)\\
     &=\gamma-\mu_{n-1}+\eta_{n}+x_{n-1}-\max\{0,\,x_{n-1}+x_1-1\}-(1-x_1-x_n)\\
     &=\gamma-x_{n}+0+x_{n-1}-(x_{n-1}+x_1-1)-(1-x_1-x_n)=\gamma.
    \end{align*}
If $1-x_1\leqslant x_n\leqslant x_{n-1}$ then 
    \begin{align*}
  \delta_n&=\delta_{n-1}\\
     &=\gamma-\mu_{n-1}+\eta_{n}+x_{n-1}-\max\{0,\,x_{n-1}+x_1-1\}\\
     &=\gamma-x_{n}+(x_n+x_1-1)+x_{n-1}-(x_{n-1}+x_1-1)=\gamma.
    \end{align*}
Finally, if $x_n\geqslant x_{n-1}$ then 
\begin{align*}
  \delta_n&=\delta_{n-1}-(x_{n}-x_{n-1})\\
     &=\gamma-\mu_{n-1}+\eta_{n}+x_{n-1}-\max\{0,\,x_{n-1}+x_1-1\}-(x_{n}-x_{n-1})\\
     &=\gamma-x_{n-1}+(x_n+x_1-1)+x_{n-1}-(x_{n-1}+x_1-1)-(x_{n}-x_{n-1})=\gamma.
    \end{align*}
  \end{description}

\item[Case 4] $n-1\in E^-,\,n\in E^+$. We may assume $\mu( X_{n-1}\setminus X_1)<x_n<1-\mu( X_{n-1}\setminus X_1)$ since otherwise $\delta_n=0$.
  \begin{description}
  \item[Case 4.1] $x_1+x_{n-1}\leqslant 1$. In this case $\delta_{n-1}=\mu( X_1\cap
    X_{n-1})$.

If $x_n\leqslant x_{1}$ then 
    \begin{align*}
  \delta_n&=\delta_{n-1}-(x_{1}-x_n)\\
     &=\gamma+\eta_{n-1}-\mu_{n}+x_1-\max\{0,\,x_{n-1}+x_1-1\}-(x_{1}-x_n)\\
     &=\gamma+0-x_{n}+x_1-0-(x_{1}-x_n)=\gamma.
    \end{align*}
If $x_{1}\leqslant x_n\leqslant 1-x_{n-1}$ then 
    \begin{align*}
  \delta_n&=\delta_{n-1}\\
     &=\gamma+\eta_{n-1}-\mu_{n}+x_1-\max\{0,\,x_{n-1}+x_1-1\}\\
     &=\gamma+0-x_{1}+x_1-0=\gamma.
    \end{align*}
Finally, if $x_n\geqslant 1-x_{n-1}$ then 
\begin{align*}
  \delta_n&=\delta_{n-1}-(x_{n-1}+x_{n}-1)\\
     &=\gamma+\eta_{n-1}-\mu_{n}+x_1-\max\{0,\,x_{n-1}+x_1-1\}-(x_{n-1}+x_{n}-1)\\
     &=\gamma+(x_{n-1}+x_n-1)-x_{1}+x_1-0-(x_{n-1}+x_{n}-1)=\gamma.
    \end{align*}

  \item[Case 4.2] $x_1+x_{n-1}> 1$. In this case $\delta_{n-1}=\mu([0,1]\setminus(X_1\cup
    X_{n-1}))$.

If $x_n\leqslant 1-x_{n-1}$ then 
    \begin{align*}
  \delta_n&=\delta_{n-1}-(1-x_{n-1}-x_n)\\
     &=\gamma+\eta_{n-1}-\mu_{n}+x_1-\max\{0,\,x_{n-1}+x_1-1\}-(1-x_{n-1}-x_n)\\
     &=\gamma+0-x_n+x_1-(x_{n-1}+x_1-1)-(1-x_{n-1}-x_n)=\gamma.
    \end{align*}
If $1-x_{n-1}\leqslant x_n\leqslant x_{1}$ then 
    \begin{align*}
  \delta_n&=\delta_{n-1}\\
     &=\gamma+\eta_{n-1}-\mu_{n}+x_1-\max\{0,\,x_{n-1}+x_1-1\}\\
     &=\gamma+(x_{n-1}+x_n-1)-x_{n}+x_1-(x_{n-1}+x_1-1)=\gamma.
    \end{align*}
Finally, if $x_n\geqslant x_{1}$ then 
\begin{align*}
  \delta_n&=\delta_{n-1}-(x_{n}-x_{1})\\
     &=\gamma+\eta_{n-1}-\mu_{n}+x_1-\max\{0,\,x_{n-1}+x_1-1\}-(x_{n}-x_{1})\\
     &=\gamma+(x_{n-1}+x_n-1)-x_1+x_1-(x_{n-1}+x_1-1)-(x_{n}-x_{1})=\gamma.\qedhere
    \end{align*}
  \end{description}
\end{description}

\end{document}